\documentclass[12pt]{article}
\usepackage{amsmath}
\usepackage{amsfonts}
\usepackage{amssymb}
\usepackage{theorem} 

\title{The generalized continuous model theory, Borel complexity and stability}
\author{A. Ivanov}
\date{ } 

\newtheorem{thm}{Theorem}[section]
\newtheorem{lem}[thm]{Lemma}
\newtheorem{definition}[thm]{Definition}
\newtheorem{cor}[thm]{Corollary}
\newtheorem{prop}[thm]{Proposition}
\newtheorem{remark}[thm]{Remark}  
\newtheorem{example}[thm]{Example}

\begin{document}
\maketitle
\topskip 20pt

\begin{quote}
{\bf Abstract.} 
Given  Polish space $\mathcal{Y}$ and a continuous language $L$ 
we study the corresponding logic $\mathsf{Iso}(\mathcal{Y})$-space $\mathcal{Y}_L$. 
We build a framework of generalized model theory towards analysis of Borel complexity of families of subsets of Effros spaces  
$\mathcal{F}(\mathcal{Y})^k_L\times  \mathcal{F}(\mathsf{Iso} (\mathcal{Y}))^l$ corresponding to standard model-theoretic properties. 
In this paper we mainly apply this approach to stability. 
\bigskip


\end{quote}

\tableofcontents 

\section{Introduction}  
\paragraph{Generalized model theory.} 
Let $L = ( R_i )_{i\in I}$ be a countable relational language where 
$R_i$ is an $n_i$-ary relational symbol. 
The product space $X_L= \prod_{i\in I} 2^{\omega^{n_i}}$ can be viewed  
as the space of $L$-structures on $\omega$. 
If $F$ is a countable fragment of $L_{\omega_1 \omega }$, 
then the family of all sets 
$\mathsf{Mod}(\phi,\bar{s}) = \{ M\in X_L :M\models \phi(\bar{s})\}$ 
for $\phi \in F$ and tuples $\bar{s} \in \omega$, forms a basis of a topology on $X_L$ denoted by $t_F$.  
The logic action of the group $S_{\infty}$ of all permutations of $\omega$ on $X_L$ is continuous with respect to $t_F$. 
When one fixes a closed subgroup $G<S_{\infty}$ and a structure, say 
$M_0$ on $\omega$ with $G=\mathsf{Aut}(M_0 )$ then a similar topology can be defined on the $G$-space of all $L$-expansions of $M_0$.  
Becker has noticed in \cite{becker} and \cite{becker2} that some basic 
model theoretic concepts and theorems can be formulated in topological terms concerning this $G$-space.   
Thus it becomes natural to consider the corresponding properties in the general case of Polish $G$-spaces. 
This approach is called {\em generalized model theory}. 
Imitating the construction of $t_F$, Becker introduces two central notions of this theory:  {\em nice basis and nice topology}. 
Many theorems of traditional model theory can be 
viewed as topological statements concerning 
spaces with nice topologies. 
Majcher-Iwanow extended Becker's approach in \cite{MI05} 
and \cite{MI08} to some other model theoretic and descriptive set theoretic issues, for example model completeness and  $\aleph_0$-categoricity.   
\parskip0pt 

In order to remove the assumption of Becker that $G\le S_{\infty}$,   
Ivanov and Majcher-Iwanow have applied in \cite{IMI-APAL} some methods of  continuous model theory.  
We now give a brief sketch of this idea. 
In the next section it will be described in detail.  

Let $(\mathcal{Y},d)$ be a Polish space and $\mathsf{Iso}(\mathcal{Y})$ 
be the corresponding isometry group 
endowed with the pointwise convergence topology. 
Then $\mathsf{Iso} (\mathcal{Y})$ is a Polish group. 
This will be a counterpart of $S_{\infty}$ above. 
For any countable continuous signature $L$ the set 
$\mathcal{Y}_L$ of all continuous metric $L$-structures on 
$(\mathcal{Y},d)$ can be considered as a Polish 
$\mathsf{Iso}(\mathcal{Y})$-space, which is also called the {\em logic  action}. 
As in the discrete case, fixing a closed subgroup 
$G\le\mathsf{Iso}(\mathcal{Y})$ and a structure $M_0$ on $\mathcal{Y}$ 
with $G=\mathsf{Aut}(M_0 )$, 
the $G$-space of all $L$-expansions of $M_0$ arises. 
It is crucial that by a result of Melleray from \cite{melleray} 
every Polish group $G$ can be realized as the automorphism 
group of a continuous metric structure on
an appropriate Polish space  $(\mathcal{Y},d)$.  
Furthermore, it is proved in \cite{CL} and \cite{IMI-APAL} 
that given closed $G\le \mathsf{Iso}(\mathcal{Y})$ 
the corresponding logic action is universal for Borel reducibility 
of orbit equivalence relations of Polish $G$-spaces.  
  
Note that for any tuple $\bar{s}\in \mathcal{Y}$ the map 
$g\rightarrow d(\bar{s},g(\bar{s}))$ can be considered 
as a {\em grey}/fuzzy subgroup of $G$. 
Grey subsets  and subgroups of metric spaces (groups) 
were introduced in \cite{BYM} as {\em graded} counterparts 
of subsets and subgroups. 
The map $g\rightarrow d(\bar{s},g(\bar{s}))$ is called 
the {\em grey stabilizer} of $\bar{s}$.  
For any continuous sentence $\phi$ the map 
$\mathcal{Y}_L \rightarrow [0,1]$ defined by assigning to $M$ 
the interpetation $\phi^M$,  
is a grey subset of  $\mathcal{Y}_L$.

Let us now start with an arbitrary Polish group $G$ and 
a Polish $G$-space $\mathcal{X}$. 
Distinguishing a family of grey subgroups of $G$ and an appropriate family 
$\mathcal{B}$ of grey subsets of $\mathcal{X}$ we arrive at 
the situation very similar to the one described above for the logic space $\mathcal{Y}_L$.  
For example, we can treat elements of $\mathcal{B}$ as 
continuous formulas. 
The conditions on $\mathcal{B}$ under which this approach allows simulation of traditional/continuous model theory
are precisely formulated in \cite{IMI-APAL}. 
They lead to 
{\em continuous versions of nice bases and nice topologies}. 
Then typical notions/properties naturally arising for logic actions  
can be applied in the general case of a Polish 
$G$-spaces. 
Furthermore, the existence theorem from \cite{IMI-APAL}  
states that under natural circumstances appropriately 
defined nice topologies can be always found. 
This gives the {\em continuous version of generalized model theory}.  

\paragraph{Towards Borel complexity. The Effros space. }
The aim of this paper is to demonstrate that the tools 
of generalized model theory nicely work for some  
other aspects of logic actions.  
To describe it let us fix a Polish space $\mathcal{Y}$. 
Then for each pair of natural numbers $k$ and $l$ the space 
$\mathcal{Y}^k_L\times  (\mathsf{Iso} (\mathcal{Y}))^l$ is 
also Polish. 
Any typical property of structures (automorphisms) defines a subset of 
$\mathcal{Y}^k_L\times  (\mathsf{Iso} (\mathcal{Y}))^l$.  
The question, when this subset is Borel (and what the Borel class of it is)  
was studied in Sections 3 - 4 of \cite{IMI-Arx}. 

In Sections 3 - 4 of the present paper we consider other applications of our approach. 
In particular, in Section 4 instead of the space 
$\mathcal{Y}^k_L\times  (\mathsf{Iso} (\mathcal{Y}))^l$ 
we now consider the Effros Borel structure on 
$\mathcal{F}(\mathcal{Y}_L )^k \times  \mathcal{F}(\mathsf{Iso} (\mathcal{Y}))^l$, 
where  $\mathcal{F}(\mathcal{X})$ is the set of closed subsets of 
$\mathcal{X}$. 
Here an elements of $\mathcal{F}(\mathcal{X})$ is a counterpart of a theory. 
We will see that some natural logic properties (mainly stability) 
appear in this space as topological objects involving  nice/good topologies.  
This gives a tool for evaluation of Borel complexity of these notions. 

Stability in the context of generalized model theory is the central topic of the paper. 
In Section 3 we consider stability in the case of the logic space $\mathcal{Y}_L$ (Definition \ref{preunst}). 
On the one hand, this notion obviously imitates the standard one, but on the other hand, it is formulated to be independent on the data (i.e. structure) given on $\mathcal{Y}$.  
This also leads us to the notion of relative stability.  
This is the main object of Section 3. 
In Sections 3.2 and 3.3 we consider some basic examples of this notion. 
In particular, in Section 3.2 we present an expansion of the Urysohn space $\mathfrak{U}$ by a kind of the generic subspace. 
We show that it is relatively stable. 
This is the most technical part of the paper.  
On the other hand, in Section 3.3 we will see that even in the classical case of first-order stability the corresponding relative notions can be nicely adapted and applied. 
\bigskip 

The paper is basically self-contained. 
It does not assume any special background.

\section{Logic space of continuous structures} 

The subject of our paper can be situated between 
Invariant Descriptive Set Theory (\cite{bk}, \cite{gao-book}, \cite{kechris}) 
and  Continuous Model Theory (\cite{BYBHU}, \cite{BNT}, \cite{BYU}). 
The definitions below can be found in these sources. 

\subsection{Polish group actions.}

A {\em Polish space (group)} is a separable, completely
metrizable topological space (group).
Sometimes we extend the corresponding metric to 
tuples by 
$$
d((x_1 ,...,x_n ), (y_1 ,...,y_n ))= \mathsf{max} (d(x_1 ,y_1 ),...,d(x_n ,y_n )). 
$$ 
If a Polish group $G$ continuously acts on a Polish space $\mathcal{X}$,
then we say that $\mathcal{X}$ is a {\em Polish $G$-space}. 
We  say that a subset of $\mathcal{X}$ is {\em invariant} if
it is $G$-invariant. \parskip0pt

Let $(\mathcal{Y},d)$ be a Polish space and $\mathsf{Iso}(\mathcal{Y},d)$ 
be the corresponding isometry group 
endowed with the pointwise convergence toplogy. 
Then $\mathsf{Iso} (\mathcal{Y},d)$ is a Polish group. 
We fix a compatible left-invariant metric $\rho$ on it and a dense countable set $\Upsilon \subset \mathsf{Iso}(\mathcal{Y},d)$. 
In any closed subgroups of $\mathsf{Iso}(\mathcal{Y},d)$ we distinguish 
the base consisting of all sets of the form 
$N_{\sigma ,q} = \{ \alpha : \rho (\alpha, \sigma )<q \}$, $\sigma \in \Upsilon$ 
and $q\in \mathbb{Q}$. 
We may assume that $\Upsilon$ is a subgroup of $\mathsf{Iso} (\mathcal{Y})$. 
(We will often omit $d$ in this expression.)

\subsection{Continuous structures.} 

The papers \cite{BYBHU} and \cite{BYU} are the main sources 
on continuous model theory which we use.   
The terms {\em continuous signature}, {\em metric $L$-structure}, 
{\em continuity moduli} and {\em formulas of continuous logic} 
are taken from these papers. 

{\em Statements} concerning metric structures are usually 
formulated in the form 
\[
\phi = 0 , 
\] 
where $\phi$ is a formula. 
Statements are often called {\em conditions} or $\bar{x}$-{\em conditions} 
(when $\phi$ depends on $\bar{x}$).  
If $\mathcal{K}$ is a class of continuous $L$-structures then 
$Th(\mathcal{ K})$ denotes the set of all conditions without free variables 
which hold in all structures of $\mathcal{ K}$. 
  
\begin{remark} \label{rel} 
{\em By Lemma 4.1 of \cite{BNT} each $n$-ary functional symbol $F$ can be replaced 
by the predicate  $D_F (\bar{x},y) = d(F(\bar{x}), y)$. 
It is clear that the continuity moduli with respect to variables from $\bar{x}$ are the same 
and $\mathsf{id}$ works as a continuity modulus for $y$. 
Thus we may always assume that $L$ is relational. }
\end{remark}

We sometimes replace conditions of the form $\phi \dot{-} \varepsilon =0$ 
where $\varepsilon \in [0,1]$ by more convenient 
expressions $\phi \le \varepsilon$.    

We often extend the set of formulas by the application 
of {\em truncated products} by positive rational numbers.   
This means that when $q\cdot x$ is greater than $1$, the 
truncated product of $q$ and $x$ is $1$. 
Since the context is always clear, we preserve the same notation $q\cdot x$. 
The continuous logic after this extension does not differ from 
the basic case. 

It is worth noting that the choice of the set of connectives as in \cite{BYBHU} and \cite{BYU}
guarantees that for any continuous relational structure $M$, 
any formula $\phi$ is a $\gamma$-uniform 
continuous function from the 
appropriate power of $M$ to $[0,1]$, where 
$\gamma (\varepsilon )$ is of the form  
$$ 
\frac{1}{n} \cdot  
\mathsf{min} \{ \gamma ' (\varepsilon ) : \gamma ' 
\mbox{ is a continuity modulus of an } 
L\mbox{-symbol appearing in the formula} \} , 
$$ 
where the number $n$ only depends on the complexity of  $\phi$. 
This follows from the fact that assuming $\phi_1$ and $\phi_2$ have 
continuity moduli $\gamma_1$ and $\gamma_2$ respectively, the formula 
$f(\phi_1 ,\phi_2 )$ obtained by applying a binary connective $f$ has continuity modulus 
$\mathsf{min} (\gamma_1 (\frac{1}{2} x) , \gamma_2 (\frac{1}{2} x) )$. 

It is observed in Appendix A of \cite{BYU} that instead of continuity moduli 
one can consider {\em inverse continuity moduli}. 
Slightly modifying that place in \cite{BYU} we define it as follows. 
\begin{itemize} \label{inverse} 
\item A continuous monotone function $\delta :[0,1]\rightarrow [0,1]$ with $\delta(0) =0$ 
is an inverse  continuity modulus of a map $F(\bar{x}) :  \mathcal{X}^n \rightarrow [0,1]$ 
if for any $\bar{a}$, $\bar{b}$ from $\mathcal{X}^n$, 
$$ 
|F(\bar{a} )-F(\bar{b})| \le \delta (d(\bar{a} ,\bar{b} )). 
$$ 
\end{itemize}  
In \cite{IMI-APAL} this notion was applied for the construction of certain subgroups of isometries, associated to continuous sentences. 
This material will be recalled in Sections 2.1 - 2.2. 
The choice of the connectives above guarantees that 
the following statement holds, see  \cite{IMI-APAL}.  

\begin{lem} \label{ContMod} 
For any continuous relational structure $M$, 
where each $n$-ary relation has $n\cdot \mathsf{id}$ as its inverse 
continuity modulus, every formula $\phi$ admits an inverse continuity 
modulus which is of the form $k \cdot \mathsf{id}$, where $k$ 
depends on the complexity of $\phi$. 
\end{lem} 

For a continuous structure $M$ defined on $(\mathcal{Y},d)$ let 
$\mathsf{Aut}(M)$ be the subgroup of $\mathsf{Iso} (\mathcal{Y})$ consisting of 
all isometries preserving the values of atomic formulas. 
It is easy to see that $\mathsf{Aut}(M)$ is a closed subgroup with 
respect to the topology on $\mathsf{Iso} (\mathcal{Y})$ defined above. 

For every $c_1 ,\ldots ,c_n \in M$ and $A\subseteq M$ 
we define the $n$-type $\mathsf{tp}(\bar{c}/A)$ of $\bar{c}$ over $A$ 
to be the set of all $\bar{x}$-conditions with parameters from $A$ 
which are satisfied by $\bar{c}$ in $M$.  
Let $S_n (T_A )$ be the set of all $n$-types over $A$ 
of the expansion of the theory $T$ by constants from $A$. 
There are two natural topologies on this set. 
The {\em logic topology} is defined by the basis consisting of 
sets of types of the form $[\phi (\bar{x})<\varepsilon ]$, 
i.e. types containing some $\phi (\bar{x})\le \varepsilon'$ with 
$\varepsilon '<\varepsilon$.   
The logic topology is compact. 

The $d$-topology is defined by the metric 
$$
d(p,q)= \mathsf{inf} \{  d(\bar{c} ,\bar{b})| \mbox{ there is a model } M \mbox{ with } M\models p(\bar{c})\wedge q(\bar{b})\}. 
$$ 
By Propositions 8.7 and 8.8 of \cite{BYBHU} the $d$-topology is finer 
than the logic topology and $(S_n (T_A ),d)$ is a complete space. 

The following notion is helpful when we study some concrete 
examples, for example the Urysohn space. 
A relational continuous structure $M$ is 
{\em approximately ultrahomogeneous} if for any $n$-tuples 
$(a_1 ,\ldots ,a_n )$ and $(b_1 ,\ldots ,b_n )$ with the same 
quantifier-free type (i.e. with the same values of 
predicates for corresponding subtuples) and any 
$\varepsilon >0$ there exists $g\in \mathsf{Aut}(M)$ such that 
$$ 
\mathsf{max} \{ d(g(a_j ),b_j ): 1\le j \le n\} \le \varepsilon . 
$$  
As we already mentioned any Polish group can be chosen 
as the automorphism group of a continuous metric 
structure which is approximately ultrahomogeneous.  

The bounded Urysohn space $\mathfrak{U}$ (see Section 2.3) 
is {\em ultrahomogeneous} in the traditional sense: 
any partial isomorphism between two tuples extends 
to an automorphism of the structure \cite{U}.  
Note that this obviously implies that $\mathfrak{U}$ is 
approximately ultrahomogeneous. 

\bigskip 

In some places of the paper we use continuous 
$L_{\omega_1 \omega}$-logic. 
It extends the first-order logic by new connectives applied to 
countable families of formulas : $\bigvee$ is the infinitary 
$\mathsf{min}$ and $\bigwedge$ corresponds to the infinitary $\mathsf{max}$.  
The version of $L_{\omega_1 \omega}$ studied in \cite{BYIov} (see also \cite{BNT}) has a restriction that 
when these connectives are applied we demand that the formulas of the family all obey the same continuity modulus.

\subsection{Logic action} \label{LA}

Fix a countable continuous signature 
\[
L=\{ d,R_1 ,\ldots ,R_k ,\ldots , F_1 ,\ldots , F_l ,\ldots \}
\] 
and a Polish space $(\mathcal{Y},d)$. 
Let $S$ be a dense countable subset of $\mathcal{Y}$. 
Let $\mathsf{seq}(S)=\{ \bar{s}_i :i\in \omega\}$ be the set (and an enumeration) 
of all finite sequences (tuples) from $S$. 
Let us define the space of metric $L$-structures on $(\mathcal{Y},d)$. 
We introduce a metric on the set of $L$-structures as follows 
\footnote{in fact, using the recipe of the standard metric on $\mathsf{Iso}(\mathcal{Y})$}. 
Enumerate all tuples of the form $(\varepsilon ,j,\bar{s})$, where  
$\varepsilon \in \{ 0,1\}$ and when $\varepsilon =0$,  $\bar{s}$ is a tuple 
from $\mathsf{seq}(S)$ of the length of the arity of $R_j$,  and for $\varepsilon =1$, 
$\bar{s}$ is a tuple from $\mathsf{seq}(S)$ of the length of the arity of $F_j$.  
For metric $L$-structures $M$ and $N$ let   
\[ 
\delta_{\mathsf{seq}(S)} (M ,N )= \sum_{i=1}^{\infty} \{ 2^{-i} |R^M_j (\bar{s} )-R^N_j (\bar{s} )| \, \, :  \, \, 
\footnote{resp. $ 2^{-i} d(F^M_j (\bar{s} ),F^N_j (\bar{s} ))$ when $\varepsilon =1$}  
\mbox{ : }  i \mbox{ is the number of } (\varepsilon ,j,\bar{s}) \} . 
\] 
Since the predicates and functions are uniformly continuous 
(with respect to moduli of $L$) and $S$ is dense 
in $\mathcal{Y}$, we see that $\delta_{\mathsf{seq}(S)}$ is a complete metric. 
Furthermore, by an appropriate choice of rational values for 
$R_j (\bar{s})$ we find a countable dense subset of metric 
structures on $\mathcal{Y}$, i.e. the space obtained is Polish. 
We denote it by $\mathcal{Y}_L$. 
It is clear that $\mathsf{Iso} (\mathcal{Y})$ acts on $\mathcal{Y}_L$ continuously. 
Thus we consider $\mathcal{Y}_L$ as an $\mathsf{Iso}(\mathcal{Y})$-space and 
call it the {\em space of the logic action} on $\mathcal{Y}$.

It is convenient to consider the following {\em basis} 
of the topology of $\mathcal{Y}_L$. 
Fix a finite sublanguage $L'\subset L$, a finite subset 
$S'\subset S$, a finite tuple $q_1 ,\ldots ,q_t \in \mathbb{Q} \cap [0,1]$ 
and a rational $\varepsilon \in [0,1]$ with $1-\varepsilon < 1/2$.  
Consider a finite diagram $D$ of $L'$ on $S'$ of some inequalities of the form   
\[
d(F_j (\bar{s}) ,s' ) > \varepsilon \mbox{ , } d(F_j (\bar{s}) ,s' )< 1 - \varepsilon , 
\]
\[ 
|R_j (\bar{s}) - q_i | > \varepsilon \mbox{ , } |R_j (\bar{s}) - q_i |< 1 - \varepsilon ,\mbox{ with } \bar{s}\in \mathsf{seq}(S'), s'\in S'  
\]
(i.e. in the case of relations we consider 
negations of statements of the form: 
$|R_j (\bar{s}) - q_i |\le \varepsilon$ , 
$|R_j (\bar{s}) - q_i |\ge 1 - \varepsilon$).  
The set of metric $L$-structures realizing $D$ is 
an open set of the  topology of $\mathcal{Y}_L$, and 
the family of sets of this form is a basis of this topology.   
Compactness theorem for continuous logic (see \cite{BYU}) 
shows that the topology is compact. 
We will also call it the {\em logic topology}. 

We also mention Theorem 2.2 from \cite{CL} that 
for every Polish group $G$ and every its Borel action on a standard Borel space 
$\mathcal{X}$ there exists an $L_{\omega_1\omega}$-sentence $\phi$ 
and a continuous embedding of $G$ into some $\mathsf{Iso}(\mathcal{Y})$ 
such that the action of $G$ on $\mathcal{X}$ arises 
as the action of $G$ on the subspace of $\mathcal{Y}_L$ of models of $\phi$. 
It is a consequence of the continuous version of the Lopez-Escobar theorem 
(\cite{BNT}, \cite{CL}). 
This issue will be also discussed in Section 2.4. 

\section{Good  (almost nice) topologies} 

In Section 2.2 we give the main concepts of 
the generalized model theory. 
They are based on the notion of a grey subset. 
The corresponding preliminaries are given in Section 2.1. 
In Section 2.3 we describe the most important examples arising  
in generalized continuous model theory. 

It is worth mentioning that the crucial notions of generalized model theory are nice topologies and nice bases, see  \cite{becker2} and \cite{IMI-APAL}. 
In the present paper we do not give the definition of a nice basis in detail. 
The reason is that the arguments where this term appears, do not use the definition in full. 
Usually these arguments can be realized for the weaker notion of {\em good bases}.  

\subsection{Grey subsets} 

The notion of grey subsets was introduced in \cite{BYM} and applied in \cite{BNT}, \cite{CL}, \cite{IMI-APAL}, \cite{Chen}. 
 
A {\em grey subset} of a space $\mathcal{X}$, denoted 
$\phi \sqsubseteq \mathcal{X}$, is a function from $\mathcal{X}$ to  $[0,\infty ]$. 
\begin{itemize} 
\item For $r\in \mathbb{R}$ the set  $\phi_{<r} = \{ z\in \mathcal{X}: \phi (z) <r \}$ is called a {\em  cone} of $\phi$. 
Cones $\phi_{\le r}$ (closed cone) and $\phi_{> r}$, $\phi_{\ge r}$ are defined in a similar way. 
\item The grey subset $\phi$ is {\em open }(resp. {\em closed)}, 
$\phi \sqsubseteq_o \mathcal{X}$ (resp. $\phi \sqsubseteq_c \mathcal{X}$), 
if the cone $\phi_{<r}$ (resp. $\phi_{>r}$) is open for all $r\in \mathbb{R}$. 
\end{itemize} 

We also write $\phi \in {\bf \Sigma}_1$ when 
$\phi \sqsubseteq_o \mathcal{X}$ and we write 
$\phi \in {\bf \Pi}_1$ when $\phi \sqsubseteq_c \mathcal{X}$.  
We will assume below that values of 
a grey subset belong to $[0,1]$. 
Then Borel classes of grey sets are defined as follows. 
\begin{itemize} 
\item $\phi \in  {\bf \Pi}_{\alpha}$ iff $1- \phi \in {\bf \Sigma}_{\alpha}$ and
\item $\phi \in {\bf \Sigma}_{\alpha}$ iff $\phi = \mathsf{inf}_n \phi_n$ 
where $\phi_n \in \bigcup_{\beta <\alpha} {\bf \Pi}_{\beta}$.
\end{itemize} 

Let us return to the situation of Section \ref{LA}. 
We fix a language $L$, a countable dense subset $S$ of $\mathcal{Y}$ 
and study subsets of  $\mathcal{Y}_L$. 
One of the basic observations is that any first-order continuous sentence  
$\phi (\bar{c})$, $\bar{c}\in S$, defines a grey subset of $\mathcal{Y}_L$: 
$$
\phi (\bar{c}) \mbox{ takes } M \mbox{ to the value } \phi^{M} (\bar{c}). 
$$ 
Moreover, Proposition \ref{EsLo} below says that   
$\phi (\bar{c})$ defines a grey subset of $\mathcal{Y}_L$ which belongs 
to ${\bf \Sigma}_n$  for some $n$.  
It is Proposition 1.1 in \cite{IMI-APAL}. 

\begin{prop} \label{EsLo}
For any continuous formula $\phi(\bar{v})$ of the 
language $L$ there is a natural number $n$ such that 
for any tuple $\bar{a}\in S$ 
and $\varepsilon \in [0,1]$, the subset 
$$
Mod(\phi ,\bar{a},<\varepsilon )=\{ M :M \models\phi (\bar{a})<\varepsilon \} 
$$
$$
\mbox{ ( or } 
Mod(\phi ,\bar{a},>\varepsilon )=\{ M :M \models\phi (\bar{a})>\varepsilon \} \mbox{ ) } 
$$ 
of the space $\mathcal{Y}_L$ of $L$-structures, belongs to ${\bf \Sigma}_n$. 
\end{prop}

When $G$ is a Polish group, then a grey subset 
$H \sqsubseteq G$ is called a {\em grey subgroup} if 
$$
H(1)=0 \mbox{ , }\forall g\in G (H(g)=H(g^{-1})) \mbox{ and } 
\forall g, g'\in G (H(gg')\le H(g)+H(g')). 
$$ 
This is equivalent to Definition 2.5 from \cite{BYM}. 
It is worth noting that by Lemma 2.6 of \cite{BYM} 
an open grey subgroup is clopen. 

\bigskip 

If $H$ is a grey subgroup, then for every $g\in G$ we define 
the grey coset $Hg$ and the grey conjugate $H^g$ as follows:
$$
\begin{array}{l@{\ = \ }l}
Hg(h)&H(hg^{-1})\\
H^g(h)&H(ghg^{-1}).
\end{array}
$$
Observe that if $H$ is open, then $Hg$ is an open grey subset and
$H^g$ is an open grey subgroup.

\begin{definition} 
Let $\mathcal{X}$ be a continuous $G$-space.  
A grey subset $\phi \sqsubseteq \mathcal{X}$ is called  {\sl invariant} with 
respect to a grey subgroup $H \sqsubseteq G$ if 
for any $g\in G$ and $x\in \mathcal{X}$ we have $\phi (g(x)) \le \phi (x) \dot+ H(g)$. 
\end{definition}  

Since $H(g)=H(g^{-1})$, the inequality from the definition is 
equivalent to $\phi (x) \le \phi (g(x)) \dot+ H(g)$. 

\bigskip 

\begin{remark} {\em (see Section 2.1 of \cite{IMI-APAL}).
It is clear that for every continuous structure $M$ 
(defined on $\mathcal{Y}$) every continuous formula $\phi (\bar{x})$ 
defines a clopen grey subset of $M^{|\bar{x}|}$. 
Furthermore, note that when $\phi (\bar{x},\bar{c})$ is a continuous 
formula with parameters $\bar{c}\in M$ and $\delta$ is a linear inverse 
continuous modulus for $\phi (\bar{x}, \bar{y})$ (see Definition \ref{inverse}), 
then $\phi$ is invariant with respect to the open grey subgroup 
$H_{\delta, \bar{c}} \sqsubseteq \mathsf{Aut}(M)$ defined by  
$$
H_{\delta, \bar{c}}(g) =\delta ( d((c_1 ,\ldots ,c_n ), (g(c_1 )),\ldots ,g(c_n ))) \mbox{, where } g\in \mathsf{Aut}(M) ,
$$
i.e.  
$$
\phi(g(\bar{a}),\bar{c}) \le \phi (\bar{a},\bar{c}) + H_{\delta ,\bar{c}}(g) .    
$$ 
}  
\end{remark} 
\bigskip 
 
In the space of continuous $L$-structures $ \mathcal{Y}_L$ this 
remark has the following version (see Lemma 2.2 in \cite{IMI-APAL}). 

\begin{lem} 
Let $\delta$ be an inverse continuity modulus for $\phi (\bar{x})$, 
which is linear. 
Let the {\sl grey stabilizer}    
$H_{\delta ,\bar{c}}\sqsubseteq \mathsf{Iso}(\mathcal{Y})$ be defined 
as follows: 
\[
H_{\delta ,\bar{c}}(g) =\delta ( d((c_1 ,\ldots ,c_n ), (g(c_1 )),\ldots ,g(c_n ))) \mbox{, where } g\in \mathsf{Iso}(\mathcal{Y}) . 
\] 
Then the grey subset defined by $\phi (\bar{c}) \sqsubseteq \mathcal{Y}_L$ 
is invariant with respect to  $H_{\delta ,\bar{c}}$. 
\end{lem} 

\subsection{Good bases}

We start with the following definition. 
 
\begin{definition}   \label{gbasis} 
A family $\mathcal{U}$  of open grey subsets of a Polish space 
$\mathcal{X}$ with a topology $\tau$ is called a {\sl grey basis}   
of $\tau$ if the family 
$\{\phi_{<r}:\phi\in{\mathcal{U}}, r\in {\mathbb{Q}}\cap (0,1)\}$ 
is a basis of $\tau$.
\end{definition} 

We now describe our typical assumptions on $G$:  
\begin{itemize} 
\item $G$ is a Polish group; 
\item we distinguish a countable dense subgroup $G_0 <G$ and 
a countable family of clopen grey subsets $\mathcal{R}$ of $G$ 
which is a grey basis of the topology of $G$; 
\item we assume that $\mathcal{R}$ consists of all  
$G_0$-cosets of grey subgroups from $\mathcal{R}$, i.e. 
for each $\rho\in\mathcal{R}$ there is a grey subgroup 
$H\in \mathcal{R}$ and an element $g_0 \in G_0$ 
so that for any $g\in G$, $\rho (g ) = H(g g^{-1}_0 )$; 
\item  we assume that $\mathcal{R}$ is closed under 
$G_0$-conjugacy, under $\mathsf{max}$ and truncated 
multiplication by positive rational numbers. 
\end{itemize}  

\begin{remark}\label{Gbasis}  
{\em In Remark 2.9 of \cite{IMI-APAL}
it is observed that for every Polish group $G$ there is a a countable $G_0 <G$ and 
a countable family of open grey subsets $\mathcal{R}$ 
satisfying these assumptions. }
\end{remark} 

When we fix $G_0$, $\mathcal{R}$ and consider 
a Polish $G$-space $(\mathcal{X},d)$
we also distinguish a countable grey basis  
$\mathcal{U}$ of the topology of $\mathcal{X}$.  
Let $\tau$ be the corresponding topology.  

In  generalized model theory the space $(\mathcal{X}, \tau )$ 
is considered together with some new topology which is called {\em nice};   
see \cite{becker2} for the case of $G$-spaces with non-archimedian $G$.  
In the general case of Polish $G$-spaces this idea has been 
realized in \cite{IMI-APAL} using continuous logic.   
Since we do not need the corresponding material in 
exact form we give the following very general definition. 

\begin{definition}  \label{NB}
Let $\mathcal{R}$ be a grey basis of $G$ consisting of cosets of 
open grey subgroups of $G$ which also belong to $\mathcal{R}$. 
Assume that the subfamily of $\mathcal{R}$ of all open grey subgroups 
is closed under $\mathsf{max}$ and truncated multiplication by numbers from 
$\mathbb{Q}^{+}$. 

We say that a family $\mathcal{B}$ of Borel grey subsets
of the $G$-space $(\mathcal{X}, \tau )$ is a {\sl good basis}  
with respect to $\mathcal{R}$ if: \\ 
(i) $\mathcal{B}$ is countable and generates a topology finer than $\tau$;\\ 
(ii) for each $\phi\in \mathcal{B}$ there exists an open grey 
subgroup $H\in \mathcal{R}$ such that $\phi$ is $H$-invariant.
\end{definition}

It will be usually assumed that 
all constant functions $q$, $q\in \mathbb{Q}\cap [0,1]$ 
are in $\mathcal{B}$.  

\begin{definition} \label{nto}
A topology ${\bf t}$ on $\mathcal{X}$ is $\mathcal{R}$-{\sl good} for the $G$-space $\langle \mathcal{X}, \tau\rangle$ if the following conditions are satisfied.\\
(a) The topology ${\bf t}$ is Polish, ${\bf t}$ is finer than $\tau$
and the $G$-action remains continuous with respect to ${\bf t}$. \\
(b) There exists a grey basis $\mathcal{B}$ of ${\bf t}$ which is good with respect to $\mathcal{R}$. 
\end{definition} 

Nice bases and nice topologies introduced in \cite{IMI-APAL} are $\mathcal{R}$-good. 
We do not reconstruct the corresponding definitions. 
It is worth noting here that it is usually a challenge to prove that a given basis is nice.  
In opposite, in the situation of standard examples (see Section 2.3 below) 
the property that the basis is good is straightforward. 

\subsection{Countable approximating substructures} \label{CAS} 

In this section we give basic examples of good bases 
and topologies on some logic spaces.  
They are built by an application of a certain  construction to different platform spaces. 
Platforms which we consider are standard examples of Polish spaces studied in continuous logic, see \cite{BYBHU}, \cite{BYBM} and \cite{BYU}. 
The logic space based on the Urysohn sphere has been already studied in this context, see \cite{CL}, \cite{IMI-APAL}. 
 
To describe the general idea of  the construction (according to \cite{IMI-APAL} and \cite{IMI-Arx}), 
we start with the following definition from \cite{BYBM}. 
\begin{definition} 
Let $(\mathcal{M},d)$ be a Polish metric structure of countable language with the universe $\mathcal{M}$. 
We say that a (classical) countable structure $N$ is a {\sl countable approximating substructure} 
of $\mathcal{M}$ if the following conditions are satisfied: 
\begin{itemize}  
\item The universe of $N$ is a dense countable subset of
$(\mathcal{M}, d)$.  
\item Any automorphism of $N$ extends to a (necessarily unique) 
automorphism of $\mathcal{M}$, and $\mathsf{Aut}(N)$ is dense in
$\mathsf{Aut}(\mathcal{M})$.
\end{itemize} 
\end{definition} 

Let $G_0$ be a dense countable subgroup of $\mathsf{Aut}(N)$.  
We may consider it as a subgroup of $\mathsf{Aut} (\mathcal{M})$. 

{\em \underline{Family $\mathcal{R}^{\mathcal{M}}(G_0 )$}.} 
Let $\mathcal{R}_0$ be the family of all clopen 
grey subgroups of $\mathsf{Aut}(\mathcal{M})$ of the (truncated) form 
\[
H_{q, \bar{s}} : g\rightarrow q \cdot d(g(\bar{s}), \bar{s}), 
\mbox{ where } \bar{s}\subset N, \mbox{ and } q\in \mathbb{Q}^{+}.   
\] 
It is clear that $\mathcal{R}_0$ is closed under 
conjugacy by elements of $G_0$. 
Consider the closure of $\mathcal{R}_0$ under 
the function $\mathsf{max}$ and define $\mathcal{R}^{\mathcal{M}}(G_0 )$ 
to be the family of all $G_0$-cosets of grey 
subgroups from $\mathsf{max}(\mathcal{R}_0 )$. 
Then $\mathcal{R}^{\mathcal{M}} (G_0 )$ is countable and the 
family of all $(H_{q, \bar{s}})_{<l}$ where 
$H\in \mathcal{R}_0$ and $l\in \mathbb{Q}$, 
generates the topology of $\mathsf{Aut}(\mathcal{M},d)$.  
Furthermore, it is easy to see that $G_0$ and 
$\mathcal{R}^{\mathcal{M}} (G_0 )$ satisfy all the conditions of/before  
Remark \ref{Gbasis} for $\mathcal{R}$. 

{\em \underline{Family $\mathcal{B}_{\mathcal{L}}$}. } 
Let $L$ be a relational language of a countable continuous signature with 
inverse continuity moduli $\le n\cdot \mathsf{id}$ for $n$-ary relations. 
We will assume that $L$ extends the language of the structure $\mathcal{M}$. 

Let $\mathcal{L}$ be a countable fragment of $L_{\omega_1 \omega}$, 
in particular $\mathcal{L}$ be closed under first-order connectives. 
Note that inverse continuity moduli of first-order continuous formulas 
(with connectives as in Introduction) can be taken linear 
(of the form $k\cdot \mathsf{id}(x)$). 
Thus it is easy to see that  every formula of $\mathcal{L}$ 
has linear inverse continuity moduli. 
In this place we apply the material of Section 2.1. 
The next paragraph is also based on it. 

Let $\mathcal{B}_{\mathcal{L}}$ be the family of all grey subsets defined 
on the logic space $\mathcal{M}_L$ by continuous $\mathcal{L}$-sentences (with parameters) 
as follows  
$$
\phi (\bar{s}) : M\rightarrow \phi^{M} (\bar{s}), 
\mbox{ where }\bar{s}\in N \mbox{ and } \phi (\bar{x}) \in \mathcal{L} . 
$$ 
By linearity of inverse continuity moduli, it is easy to see 
that for every continuous sentence $\phi (\bar{s})$ 
there is a number $q\in \mathbb{Q}$ 
(depending on the continuity modulus of $\phi$) such that 
the grey subset as above is $H_{q, \bar{s}}$-invariant. 
As a result we have the following statement. 

\begin{quote} 
Let $\mathcal{B}_{\mathcal{L}}$ be a family of grey subsets corresponding 
to a countable continuous fragment $\mathcal{L}$ of $L_{\omega_1 \omega}$. 
Then the family $\mathcal{B}_{\mathcal{L}}$ is a 
good basis with respect to $\mathcal{R}^{\mathcal{M}} (G_0 )$. 
\end{quote} 
\bigskip 

{\bf (A) The Urysohn space. } 
Let us consider the following example. 
The Urysohn space of diameter 1  
is the unique Polish metric space of diameter 1 
which is universal and ultrahomogeneous. 
This space $\mathfrak{U}$ is considered in 
the continuous signature  $\langle d \rangle$. 

The countable counterpart of $\mathfrak{U}$ is the 
{\em rational Urysohn space of diameter 1}, $\mathbb{Q}\mathfrak{U}$, 
which is both ultrahomogeneous and universal for countable 
metric spaces with rational distances and diameter $\le 1$. 
The space $\mathfrak{U}$ is interpreted as $\mathcal{M}$ above and 
$\mathbb{Q}\mathfrak{U}$ will be our $N$. 
It is shown in Section 6.1 of \cite{BYBM} that there is 
an embedding of $\mathbb{Q}\mathfrak{U}$ into $\mathfrak{U}$ so that: \\
(i) $\mathbb{Q}\mathfrak{U}$ is an approximating substructure of $\mathfrak{U}$: 
it is dense in $\mathfrak{U}$; 
every isometry of  $\mathbb{Q}\mathfrak{U}$ extends to an isometry of  
$\mathfrak{U}$, and $\mathsf{Iso} (\mathbb{Q}\mathfrak{U})$ is dense in $\mathsf{Iso}(\mathfrak{U})$;  \\ 
(ii) for every $\varepsilon>0$, every partial isometry $h$ of  
$\mathbb{Q}\mathfrak{U}$ with domain $\{ a_1 ,...,a_n\}$ and every isometry 
$g$ of $\mathfrak{U}$ such that $d(g(a_i ),h(a_i ))<\varepsilon$ 
for all $i$, there is an isometry $\hat{h}$ of  $\mathbb{Q}\mathfrak{U}$ 
that extends $h$ and is such that for all  $x\in \mathfrak{U}$, 
$d(\hat{h}(x),g(x))<\varepsilon$.  

Let $G_0$ be a dense countable subgroup of $\mathsf{Iso}(\mathbb{Q}\mathfrak{U})$.  
By (i) we may view it as a subgroup of $\mathsf{Iso} (\mathfrak{U})$. 
We now define $\mathcal{R}^{\mathfrak{U}} (G_0 )$ by the recipe above.  
As we already know $G_0$ and 
$\mathcal{R}^{\mathfrak{U}} (G_0 )$ satisfy all the conditions of/before  
Remark \ref{Gbasis}.

Let $L$ be a relational language of a continuous signature as above. 
Let $\mathcal{L}$ be a countable fragment of $L_{\omega_1 \omega}$ and  
let $\mathcal{B}_{\mathcal{L}}$ be the family of all grey subsets defined 
by continuous $\mathcal{L}$-sentences (with parameters from $\mathbb{Q}\mathfrak{U}$) as above. 
We already know that $\mathcal{B}_{\mathcal{L}}$ is a good basis. 
It is even proved in \cite{IMI-APAL} (see Theorem 3.2) that this basis is nice.

\bigskip

{\bf (B) A separable Hilbert space.} 
We follow \cite{BYBM}, \cite{IMI-Arx} and \cite{ros09}. 
Let us consider the complex Hilbert space $l_2(\mathbb{N})$. 
Let ${\cal Q}$ denote the algebraic closure of
$\mathbb{Q}$, and consider the countable subset
${\cal Q}l_2$ of $l_2 (\mathbb{N})$ of all sequences with finite 
support and coordinates from ${\cal Q}$. 
It is shown in Section 6.2 of \cite{BYBM} (with using Section 7 of \cite{ros09}), 
that it is an approximating substructure of $l_2 (\mathbb{N})$. 
In particular, we have another pair playing the role of $(\mathcal{M},N)$. 
Since $l_2 (\mathbb{N})$ is unbounded, the authors of \cite{BYBM} consider
instead its closed unit ball, equipped with functions
$x \rightarrow \alpha x$ for $|\alpha | \le 1$ 
and
$(x, y)\rightarrow \frac{x+y}{2}$, from which
$l_2 (\mathbb{N})$  can be recovered. 

Let $G_0$ be a dense countable subgroup of ${\bf U}(\mathcal{Q}l_2 )$.  
We may view it as a subgroup of ${\bf U}(l_2 (\mathbb{N}))$. 
We now apply the procedure of 
$\mathcal{R}^{\mathcal{M}} (G_0 )$ and $\mathcal{B}_{\mathcal{L}}$. 
As a result, we obtain the 
{\em family $\mathcal{R}^{\mathbb{H}}(G_0 )$},  
a grey basis defined on ${\bf U}(l_2 (\mathbb{N}))$ and a good basis on 
of the logic space $l_2 (\mathbb{N})_{L}$ 
with respect to $\mathcal{R}^{\mathbb{H}}(G_0 )$. 
A detaled description is given in Section 2 of \cite{IMI-Arx}. 
\bigskip 

{\bf (C) The measure algebra on $[0,1]$.}
Denote by $\lambda$ the Lebesgue measure on the unit interval
$[0,1]$. 
We view its automorphism group $\mathsf{Aut}([0,1], \lambda )$ as the automorphism group of the Polish metric structure
$$
( MALG,0,1,\wedge , \vee , \neg , d), 
$$ 
where MALG denotes the measure algebra on $[0,1]$ 
and
$d(A, B) = \lambda (A \Delta B)$
(see \cite{kechris}). 
Further details of presentation of it in the form as above, can be found in Section 2 of \cite{IMI-Arx}. 

\begin{remark} 
{\em The basic continuous metric structures which appear in {\bf (A)} - {\bf (C)}, 
i.e. $\mathfrak{U}$, the unit ball of $l_2 (\mathbb{N})$ and MALG, 
are ultahomogeneous structures in the classical sense: any partial isomorphism 
between two tuples extends to an automorphism of the structure. 
This is, in particular, mentioned in Section 3.1 of \cite{BYFra}. }
\end{remark}

\section{Stability in generalized model theory} 

In this section we define stability under the circumstances of Section 2.3. 
In Section 4 we will see that it is a particular case of  a general definition under the circumstances of the generalized model theory. 

Consider the situation  given in Section 2.3. 
Let $(\mathcal{M},d)$ be a continuous metric structure with universe $\mathcal{M}$ and let 
$N$ be a countable approximating substructure  of $\mathcal{M}$. 
Let $G_0$ be a dense countable subgroup of $\mathsf{Aut}(N)$ viewed as a subgroup of $\mathsf{Aut} (\mathcal{M})$ and  
let $\mathcal{R}^{\mathcal{M}}(G_0 )$ be the corresponding countable family from $\mathsf{Aut}(\mathcal{M})$ of all $G_0$-cosets of grey subgroups from $\mathsf{max}(\mathcal{R}_0 )$ as in Section 2.3.  
We remind the reader that $\mathcal{R}_0$ consists of all clopen grey subgroups of  the (truncated) form 
\[
H_{q, \bar{s}} : g\rightarrow q \cdot d(g(\bar{s}), \bar{s}), 
\mbox{ where } \bar{s}\subset N, \mbox{ and } q\in \mathbb{Q}^{+}.   
\] 

As in Section 2.3 let $L$ be a relational language of a continuous signature with inverse continuity moduli $\le n\cdot \mathsf{id}$ for $n$-ary relations. 
We again assume that $L$ extends the language of the structure $\mathcal{M}$. 
Consider the logic space $\mathcal{M}_L$. 
Let $\mathcal{B}_{\mathcal{L}}$ be a family of grey subsets corresponding to a countable continuous fragment $\mathcal{L}$ of $L_{\omega_1 \omega}$. 
As we already know the family $\mathcal{B}_{\mathcal{L}}$ is a 
good basis with respect to $\mathcal{R}^{\mathcal{M}} (G_0 )$. 
Let {\bf t} be the corresponding good topology and let $Y$ be an invariant {\bf t}-closed subset of $\mathcal{M}_{L}$. 
We view $Y$ as an $\mathsf{Aut}(\mathcal{M})$-invariant set of $L$-expansions of $\mathcal{M}$.   

Let a formula $\phi (\bar{x}, \bar{x}')$ belong to $\mathcal{L}$ and let  $\bar{s}, \bar{s}' \in N$ be tuples of appropriate length. 
Then $\phi(\bar{s}, \bar{s}')$ is viewed as an element of $\mathcal{B}_{\mathcal{L}}$. 
How to define that $\phi(\bar{s}, \bar{s}')$ is unstable with respect to $Y$?   
We now formulate a variant of the order property for $\phi (\bar{x},\bar{x}')$. 

\begin{definition} \label{preunst} 
The grey set $\phi(\bar{s},\bar{s}')$ is {\sl unstable} with respect to $Y$ if there are rational $r_1$ and $r_2\in [0,1]$ such that $r_1 < r_2$ and for any $n$ and any $\varepsilon\in \mathbb{Q} \cap [0,1]$ 
there exist $\bar{s}_1 ,\bar{s}'_1 ,\ldots ,\bar{s}_n ,\bar{s}'_n \in N$ such that 
$d(tp^{\mathcal{M}}(\bar{s}_i \bar{s}'_j), tp^{\mathcal{M}}(\bar{s} \bar{s}'))\le \varepsilon$ for all $i,j\le n$  and  
\[ 
Y \cap \bigcap \{ (\phi (\bar{s}_i ,\bar{s}'_j) \dot{-} r_1 )_{\le \varepsilon} : i < j \}   
\cap  \bigcap \{ (r_2 \dot{-} \phi (\bar{s}_i ,\bar{s}'_j))_{\le \varepsilon} : j \le i \} \not=\emptyset. 
\] 
\end{definition} 
Note, that applying this definition in the case of first-order continuous logic, the requirement that $\bar{s}_1 ,\bar{s}'_1 ,\ldots ,\bar{s}_n ,\bar{s}'_n \in N$ can be omitted (i.e. being just in $\mathcal{M}$). 
This follows by continuity of formulas and density of $N$ in $\mathcal{M}$.  
In the case when $\mathcal{M}$ is just a pure discrete set and $Y$ is determined by a theory of some non-trivial language, 
we arrive at the standard definition of unstability of $Y$.  
On the other hand, the definition excludes the case when $Th(\mathcal{M})$ is already unstable and it is witnessed by $\phi(\bar{x}, \bar{x}')$, where $\phi$ is a formula of the language of $\mathcal{M}$. 

\begin{remark} 
{\em It is clear that Definition \ref{preunst} corresponds to the situation of Section 2.3. 
On the other hand, in Section 4 we will give Definition \ref{unst}, a definition of an unstable grey subset under the circumstances of the generalized model theory,  where the presence of the ``platform" is not assumed.   
Definition \ref{unst} is a Borel property. 
We will see that separable cateoricity of $\mathcal{M}$ implies that these definitions are equivalent. 
It is worth noting that originally the author started this research with the material of Section 4, and found Definition \ref{preunst} later as an illustration of the general theory. }
\end{remark} 

\subsection{Relative stability over a stable platform}

We now connect Definition \ref{preunst} with relative stability. 
The idea of the latter is as follows. 
Let $T$ be a first-order  continuous theory of some language $L$, $\phi(\bar{x}, \bar{x}')$ be an $L$-formula, and 
$\Theta$ be a set of $L$-formulas with no free variables outside $\bar{x}\bar{x}'$.

\begin{definition} \label{preunst0}   
We call the formula  $\phi(\bar{x},\bar{x}')$  {\sl relatively unstable with respect to} $T$ {\sl and to} $\Theta$,  
if there exist rational $r_1 < r_2$, a sequence $\bar{s}_1 \bar{s}'_1 ,\ldots ,\bar{s}_n \bar{s}'_n ,\ldots$ and 
$M\models T$ such that for any  $\theta(\bar{x},\bar{x}') \in \Theta$ 
\[ 
\lim_i \lim_j \theta^M (\bar{s}_i ,\bar{s}'_j)= \lim_j \lim_i \theta^M (\bar{s}_i ,\bar{s}'_j),  
\] 
under the assumption that one of these limits exists, but    
\[
\mathsf{min} (\lim_i \lim_j \phi^M (\bar{s}_i ,\bar{s}'_j),  \lim_j \lim_i \phi^M (\bar{s}_i ,\bar{s}'_j))< r_1  \mbox{ and }  
\]   
\[ 
\mathsf{max} (\lim_i \lim_j \phi^M  (\bar{s}_i ,\bar{s}'_j)   ,  \lim_j \lim_i \phi^M  (\bar{s}_i ,\bar{s}'_j)) > r_2 . 
\] 
\end{definition} 

We now apply this definition in the situation of generalized continuous model theory given in the beginning of Section 3. 

\begin{definition} \label{preunst2} 
Under the circumstances of Definition \ref{preunst}, let $\Theta \subset \mathcal{L}$ 
be a set of formulas with no free variables outside $\bar{x}\bar{x}'$.   
We call the formula  $\phi(\bar{x},\bar{x}')$  {\sl relatively unstable with respect to} $Y$ {\sl and to} $\Theta$,  
if there exists  rational $r_1 < r_2$, a sequence $\bar{s}_1 \bar{s}'_1 ,\ldots ,\bar{s}_n \bar{s}'_n ,\ldots$ and 
$M\in Y$ such that the conclusion of Definition \ref{preunst0} holds. 
\end{definition} 

\begin{remark} 
{\em It is easy to see that Definition \ref{preunst0} (but not \ref{preunst2}) can be simplified by replacing the inequalities  involving $r_1$ and $r_2$ by the inequality 
\[ 
\lim_i \lim_j \phi^M  (\bar{s}_i ,\bar{s}'_j)   \not=  \lim_j \lim_i \phi^M  (\bar{s}_i ,\bar{s}'_j). 
\] 
Indeed, in this situation we can apply compactness and the argument of Lemma 7.2 from \cite{BYU}. }

\end{remark} 

The property of relative (un)stability appears in the following proposition. 

\begin{prop} \label{prepre}  
Under the circumstances above (and of Section 2.3) assume that $\mathcal{M}$ is a stable  first-order continuous structure.    
Assume that $Y$ corresponds to an $\mathcal{L}$-theory of $L$-expansions of $\mathcal{M}$ and  assume that a continuous 
$L$-formula  $\hat{\phi} (\bar{x},\bar{x}' )$ is first-order and has the order property with respect to $Y$ under the standard definition.  

Then $\hat{\phi} (\bar{x}, \bar{x}')$ is relatively unstable with respect to $Y$ and the set $\Theta$ of all first order formulas 
$\theta (\bar{x}, \bar{x}')$  of the language of $\mathcal{M}$. 
Furthermore, the corresponding sequence $\bar{s}_1 \bar{s}'_1 ,\ldots ,\bar{s}_n \bar{s}'_n ,\ldots$ can be taken in $N$. 
\end{prop}

{\em Proof.} 
Since $\hat{\phi} (\bar{x}, \bar{x}')$ has the order property, there exist rational numbers $r_1 < r_2$ and a sequence 
$\bar{s}_1 \bar{s}'_1 ,\ldots ,\bar{s}_n \bar{s}'_n ,\ldots$ such that in some  $M\in Y$
\[
\mathsf{min} (\lim_i \lim_j \hat{\phi}  (\bar{s}_i ,\bar{s}'_j)   ,  \lim_j \lim_i \hat{\phi}  (\bar{s}_i ,\bar{s}'_j)) < r_1 \, \mbox { and } \,  
\] 
\[ 
\mathsf{max} (\lim_i \lim_j \hat{\phi}  (\bar{s}_i ,\bar{s}'_j)   ,  \lim_j \lim_i \hat{\phi}  (\bar{s}_i ,\bar{s}'_j)) > r_2 . 
\]    
Since $N$ is dense in $\mathcal{M}$ and $\hat{\phi}  (\bar{x}, \bar{x}')$ is continuous, this sequence can be found in $N$. 

For any formula $\theta(\bar{x},\bar{x}')$ of the language of $\mathcal{M}$ consider the values $\theta^M (\bar{s}_n ,\bar{s}'_m )$ with $n<m$. 
Applying the Ramsey theorem we obtain a subsequence of $\bar{s}_n\bar{s}'_n$ such that these values are in the same 1/2-half of $[0,1]$. 
Iterating this argument we obtain a subsequence of $\bar{s}_n\bar{s}'_n$ such that 
$\lim_n \mathsf{inf} \{ \theta^M (\bar{s}_{n} ,\bar{s}'_{m}): n<m \}$ exists and coincides with 
$\lim_n \mathsf{sup} \{\theta^M (\bar{s}_{n} ,\bar{s}'_{m}): n<m\}$ for this subsequence. 
We denote the elements of it by $\bar{s}_n\bar{s}'_n$ again.       
We see that $\lim_n \lim_m \theta^M (\bar{s}_{n} ,\bar{s}'_{m})$ exists, say $\delta$. 
By stability of $\mathcal{M}$,  
$\lim_m \lim_n \theta^M (\bar{s}_{n} ,\bar{s}'_{m}) =\delta$. 

Since the language is countable this argument can be organized in a uniform way for all formulas of the language of $\mathcal{M}$. 
Note, that the inequality 
\[
\mathsf{min} (\lim_i \lim_j \hat{\phi} ^M (\bar{s}_i ,\bar{s}'_j),  \lim_j \lim_i \hat{\phi}^M  (\bar{s}_i ,\bar{s}'_j))<r_1
\]  
and the corresponding one with $\mathsf{max}$ and $r_2$ do not change for the final subsequence.  
$\Box$ 

\bigskip 

Concerning examples of Section 2.3 it is well-known that the complex Hilbert space $l_2(\mathbb{N})$ and the measure algebra MALG are stable structures, \cite{BYBHU}. 
In particular these platforms are appropriate for the proposition. 
It is also well-known that the structure $\mathfrak{U}$ is not stable \cite{BYTs}, \cite{Conant}. 
We also mention here that it follows from \cite{Conant} that the Polish ultrametric Urysohn space for $\mathbb{Q} \cap [0,1]$ is stable.

\begin{remark} \label{rem-sta}
{\em The paper \cite{BHV} gives an example of an expansion of the complex Hilbert space $l_2(\mathbb{N})$ by a generic predicate such that the corresponding theory (called $T_{N}$) is not stable (even not $NTP_2$, i.e. not simple). 
In particular, the proposition above works in this case. }
\end{remark} 

\bigskip 

\noindent  
\begin{example}\label{Example} {\bf Stable equivalence.} 
Let $M$ be a first-order continuous metric structure and 
$\phi(\bar{x},\bar{y})$ and $\theta(\bar{x},\bar{y})$ be formulas with parameters from $M$. 
\end{example} 
\begin{definition} \label{stable-eq}
We say that $\phi (\bar{x},\bar{y})$ is {\sl stably equivalent  to $\theta(\bar{x},\bar{y})$ with respect to} $Th(M)$ if 
the formula $(\phi - \theta )(\bar{x},\bar{y})$ (viewed in $[-1,1]$) is stable.  
\end{definition} 

We now connect this notion with relative stability, see Definition \ref{preunst0}.

\begin{lem} \label{stableq} 
(a) Stable equivalence is an equivalence relation on the set of formulas. \\ 
(b) If a formula $\phi (\bar{x},\bar{y})$ is stably equivalent   to $\theta(\bar{x},\bar{y})$ with respect to $Th(M)$ then if $\phi(\bar{x},\bar{y})$ is stable with respect to $Th(M)$ and $\theta(\bar{x},\bar{y})$, and vice versa. 
\end{lem} 

{\em Proof.} 
(a) The proof is straightforward.  
Note that the set of stable formulas is closed under continuous Boolean combinations, see Lemma 8.1 of \cite{BYU}.  

(b) 
Take a sequence  
$\bar{s}_1 \bar{s}'_1 \ldots \bar{s}_n \bar{s}'_n \ldots \in M$ such that 
\[
\lim_i \lim_j \theta(\bar{s}_i ,\bar{s}'_j)= \lim_j \lim_i \theta(\bar{s}_i ,\bar{s}'_j) 
\] 
and there is one of    
$\lim_i \lim_j \phi (\bar{s}_i ,\bar{s}'_j)$  or $\lim_j \lim_i \phi (\bar{s}_i ,\bar{s}'_j)$, say the first one.   
Then 
$\lim_i \lim_j \phi (\bar{s}_i ,\bar{s}'_j) - \lim_i \lim_j \theta (\bar{s}_i ,\bar{s}'_j)$ and so $\lim_i \lim_j (\phi (\bar{s}_i ,\bar{s}'_j) - \theta (\bar{s}_i ,\bar{s}'_j))$ exist. 
By stable equivalence the limit 
$\lim_j \lim_i (\phi(\bar{s}_i ,\bar{s}'_j) -  \theta (\bar{s}_i ,\bar{s}'_j))$ exists and equals to the previous one. 
Now by existence of 
$\lim_j \lim_i \theta (\bar{s}_i ,\bar{s}'_j)$ we have exisence of 
$\lim_j \lim_i \phi(\bar{s}_i ,\bar{s}'_j)$ and the required equality.   
$\Box$

\bigskip 

Let us return to the issue stated in the beginning of Section 3.1.  
The following proposition reformulates Definition \ref{preunst} as a natural version of relative stability.

\begin{prop} \label{preunst-1-2} 
Under the assumptions of Section 3, assume that $\mathcal{M}$ is a separably categorical first-order structure, 
$Y$ is defined by a first-order $L$-theory and $\hat{\phi} (\bar{x},\bar{x}')$ is a continuous first-order formula. 
Then Definition \ref{preunst} is equivalent to the following statement. 

There exists a sequence $\bar{s}_1 \bar{s}'_1 ,\ldots ,\bar{s}_n \bar{s}'_n ,\ldots$ in $N$ such that  
for any formula $\theta (\bar{x}, \bar{x}')$ 

of the language of $\mathcal{M}$ 
\[
\lim_i \lim_j \theta^{\mathcal{M}}(\bar{s}_i ,\bar{s}'_j)= \theta^{\mathcal{M}}(\bar{s} ,\bar{s}') = \lim_j \lim_i \theta^{\mathcal{M}}(\bar{s}_i ,\bar{s}'_j)   
\] 

but there is some $M\in Y$ where 
\[ 
\lim_i \lim_j \hat{\phi}^M (\bar{s}_i ,\bar{s}'_j) \not=  \lim_j \lim_i \hat{\phi}^M (\bar{s}_i ,\bar{s}'_j) \mbox{ and one of these limits exists}. 
\]  
\end{prop} 

{\em Proof.} 
By separable categoricity of $\mathcal{M}$ the logic topology of the space of types of its theory coincides with the $d$-topology (see Section 1.2).  
Applying compactness, continuity of formulas and separable categoricity of $\mathcal{M}$ we see that Definition \ref{preunst} 
implies the condition from the final part of the formulation.  

For the contrary direction assume that the sequence $\bar{s}_1\bar{s}'_1 ,\ldots ,\bar{s}_n \bar{s}'_n ,\ldots$ witnesses the latter property. 
Taking subsequences and using compactness we may assume that there are  $r_1$ and $r_2 \in [0,1]$ with $r_1 < r_2$ and, say 
\[ 
\lim_i \lim_j \hat{\phi} (\bar{s}_i ,\bar{s}'_j) \le r_1 < r_2 \le  \lim_j \lim_i \hat{\phi} (\bar{s}_i ,\bar{s}'_j) \mbox{ in some }M\in Y  \, \, \, \, (\dagger )
\] 
(the case when the limits are on the contrary sides is also possible). 
On the other hand,  since for every $\theta(\bar{x} ,\bar{x}')$ of the language of $\mathcal{M}$ 
\[ 
\lim_i \lim_j \theta^{\mathcal{M}} (\bar{s}_i ,\bar{s}'_j) = \theta^{\mathcal{M}}(\bar{s} ,\bar{s}')  = \lim_j \lim_i \theta^{\mathcal{M}}  (\bar{s}_i ,\bar{s}'_j)  ,
\] 
for any $n$ and any $\varepsilon\in \mathbb{Q} \cap [0,1]$ 
there exist
$\bar{s}''_1 ,\bar{s}'''_1 ,\ldots ,\bar{s}''_n ,\bar{s}'''_n \in N$ satisfying $(\dagger )$ such that 
$$
d(tp^{\mathcal{M}}(\bar{s}''_i \bar{s}'''_j), tp^{\mathcal{M}}(\bar{s} \bar{s}'))\le \varepsilon \mbox{ for all }i,j\le n. 
$$  
To see this apply the fact that separable categoricity guarantees that the logic topology and the $d$-topology of the space of types are the same.  
In particular Definition \ref{preunst} holds. 
$\Box$ 

\bigskip

\begin{remark} 
{\em It is worth noting that in this formulation formulas $\theta(\bar{x}_i ,\bar{x}'_j)$ can be taken with the additional property 
$\theta(\bar{s},\bar{s}')=0$. 
This follows from the observation that if the equality with them holds for some $\theta(\bar{x}_i ,\bar{x}'_j)\pm r$, then it holds for 
$\theta(\bar{x}_i ,\bar{x}'_j)$ too. 
} 
\end{remark}

\subsection {An example of a relatively stable piece over $\mathfrak{U}$} 

Remark \ref{rem-sta} describes examples of unstable expansions of stable platforms. 
The opposite situation, when the platform is not stable, but the expansion is relatively stable, seems more interesting. 
In this section such an example is realized. 
 
Let $\mathfrak{U}$,  $\mathbb{Q}\mathfrak{U}$, 
$G_0  \le \mathsf{Iso}(\mathbb{Q}\mathfrak{U})$, 
$\mathcal{R}^{\mathfrak{U}} (G_0 )$ be as in Section 2.3.   
It is well-known that $Th(\mathfrak{U})$ is not stable, for example see \cite{BYTs}, \cite{Conant}. 
Thus $\mathfrak{U}$ does not have any expansion with a stable first-order theory.   
In this section we give an example of an expansion which defines a relatively stable complete ${\bf t}$-closed subset of $\mathcal{B}_{\mathcal{L}}$, say $Y_P$. 

We fix the continuous signature $L= (d, P)$ where $P$ is unary symbol with continuity modulus $\mathsf{id}$ and assume that $\mathcal{L}$ is the first-order fragment of continuous logic of signature $L$. 
Let $\mathcal{B}_{\mathcal{L}}$ be the nice base which corresponds to $\mathcal{L}$ (see Section 2.3). 
In order to define a representative of $Y_P$ we use the Fra\"{i}ss\'{e} method. 
We will follow to \cite{BYFra}. 

Let ${\cal K}$ be the class of all finite metric spaces $(A,d,P)$ such that there is a metric space $B\supseteq A$ of diameter $\le 1$ such that for some $B_0 \subseteq B$ the predicate $P$ extends to the function $d(x,B_0 )$ on $B$.  
We firstly need to prove that $\mathcal{K}$ satisfies the definition of a Fra\"{i}ss\'{e} class, see Definition 2.9 from \cite{BYFra}.  

\begin{lem} \label{Fra} 
$\mathcal{K}$ is a Fra\"{i}ss\'{e} class. 
\end{lem} 

{\em Proof.}
According to Definition 2.9 of \cite{BYFra} the hereditary property and the joint embedding property for metric structures are defined exactly as in the discrete case.   
Thus the class $\mathcal{K}$ obviously has HP. 

Let us verify JEP and the near amalgamation property (NAP), see Definition 2.9 and Lemma 2.8 in \cite{BYFra}.    
To verify JEP it suffices to consider only structures of the form $(B, d, P)$ where $P(x) = d(x,B_0)$ for some $B_0 \subseteq B$ (see the definition of $\mathcal{K}$ before the lemma). 
In particular such $(B,d,P)$ can be viewed as the structure  
$(B,d,\{ b\}_{b\in B_0})$.  
Having $(B,d,\{ b\}_{b\in B_0})$ and $(C,d,\{ c\}_{c\in C_0})$ define the metric space $B\, \dot{\cup} \, C$ which extends both of them such that for any $b\in B$ and $c\in C$ we have $d(b,c)=1$.  
It is easy to see that $B\, \dot{\cup} \, C$ is a metric space and the function $d(x, B_0 \, \dot{\cup} \, C_0 )$ extends the corresponding functions for $B$ and $C$. 

In order to prove NAP consider again $(B,d,\{ b\}_{b\in B_0})$ and $(C,d,\{ c\}_{c\in C_0})$ with non-empty intersection $B\cap C$, where $P(x)$ in $B$ and in $C$ are defined by $P^B (x) = d(x,B_0)$ and $P^C (x) = d(x,C_0)$ respectively and, moreover, these functions agree on in $B\cap C$. 
Apply free amalgamation to them:    
extend the metric $d$ from $B$ and from $C$ so that $d (x,y) = \mathsf{min} \{ 1, \mathsf{min} \{ d(x,z) + d(y,z) \, | \, z \in B\cap C \} \}$ for $x \in B$ and $y \in C$. 
Then we view $B_0 \cup C_0$ as a subset of the above amalgamation of $B$ and $C$.  
Consider $(B\cup C, d, B_0 \cup C_0 )$ together with the function $P(x) = d(x, B_0 \cup C_0)$. 
Since $P^B (x)$ and $P^C (x)$ agree for $x \in B\cap C$ one easily verifies that when $b\in B$ then $d(b,B_0 ) \le d(b,C_0 )$. 
In particular $P(x)$ on $B\cup C$ extends $P^B (x)$. 
Similarly $P(x)$ extends $P^C (x)$. 
As a result we see that $\mathcal{K}$ satisfies the amalgamation property. 
The property NAP follows from AP, see \cite{BYFra}.  

Let $\mathcal{K}_n$ be the subclass of $\mathcal{K}$ consisting of all spaces of size $n$. 
The elements of $\mathcal{K}_n$ are considered as $n$-tuples. 
The pseudo-metric $d^{\mathcal{K}}$ on $\mathcal{K}_n$ is  defined as follows: 
\[ 
d^{\mathcal{K}} (\bar{a}, \bar{b}) = \mathsf{inf} \{ d^C (f(\bar{a}),g(\bar{b})) \, | \, C\in \mathcal{K} \mbox{ and } f,g \mbox{ are embeddings of structures } \bar{a}, \bar{b} 
\]  
\[  
\mbox{ into } C \mbox{ respectively } \}.  
\]   
Property {\em PP} of Definition 2.9 of \cite{BYFra} states that this pseudo-metric in a Fra\"{i}ss\'{e} class must be separable and complete on $\mathcal{K}_n$ for all $n$.  
To see this in our case, consider all $2n$-tuples of the form $(\bar{a}\bar{a}',d)$ with metrics $d$.   
Note that every $(A,d,P) \in \mathcal{K}$ with $|A|=n$ can be presented as a tuple $\bar{a}$ in some $(\bar{a}\bar{a}',d)$ with distinguished $A_0 \subseteq \bar{a}\bar{a}'$, where  $P(x) = d(x, A_0)$. 
Indeed, viewing $A$ as a subspace of some extension $B$ where $P$ is realized as $d(x,B_0 )$, it suffices to choose $A_0\subseteq B_0$ realizing $P$ just for elements of $A$. 
In order to prove separability of $\mathcal{K}_n$, take all such tuple $(\bar{a}\bar{a}',d, A_0)$ which can be isometrically embedded into $\mathbb{Q}\mathfrak{U}$. 
Their $\bar{a}$-parts form a dense countable set. 
To confirm this, it is enough to show how arbitrary $(\bar{a}\bar{a}',d, A_0 )$ as above can be amalgamated with its close approximations taken from $\mathbb{Q}\mathfrak{U}$. 
Since any $(A,d,P) \in \mathcal{K}$ can be isometrically realized inside $\mathfrak{U}$, we can apply density of $\mathbb{Q}\mathfrak{U}$  and property JEP in the form presented in \cite{IMIadm}: sufficiently similar metric spaces can be amalgamated so that the distance between the corresponding elements is sufficiently small. 
Note that having $(\bar{a}\bar{a}',d, A_0)$ and its approximation $(\bar{c}\bar{c}',d, C_0)$ inside a metric space on 
$\bar{a}\bar{a}'\bar{c}\bar{c}'$ with sufficiently small $\varepsilon$ between the corresponding elements, the function $d(x, A_0 \cup C_0 )$ extends both  $d(x, A_0  )$ and $d(x, C_0 )$ on the corresponding parts. 

To see completeness, consider a $d^{\mathcal{K}}$-Cauchy sequence of the form $(\bar{a}_m , d, P_m)$, $m\in \omega$. 
We view them as tuples from $\mathfrak{U}$. 
Take a Cauchy sequence of spaces $(\bar{a}_m \bar{a}'_m, d, P_m )$, where $P_m$ is realized as $d(x,A_0)$ for some $A_0 \subseteq \bar{a}_m\bar{a}'_m$.  
Furthermore, we may assume that this $A_0$ in each $\bar{a}_m\bar{a}'_m$ consists of elements of the same indices. 
Then distances between those $(\bar{a}_m \bar{a}'_m , d, P_m)$, $n\in \omega$, is the distance between types of tuples $\bar{a}_m \bar{a}'_m$ in $\mathfrak{U}$. 
Since the distance topology on the space of types is complete we obtain that the sequence $(\bar{a}_m , d, P_m )$, $m\in \omega$, has a limit point. 
 
To see continuity of $P$ on $\mathcal{K}_n$ take a sequence $(\bar{a}_m , d, P_m )$, $m\in \omega$, converging to some $(\bar{a},d,P)$ under $d^{\mathcal{K}}$.  
Let $a_{m,i}$ be the $i$-th coordinate of $\bar{a}_m$. 
We have to show that $P_m (a_{m,i}) \rightarrow P(a_i)$. 
As above extend each $(\bar{a}_m , d, P_m )$ to appropriate 
$(\bar{a}_m \bar{a}'_m, d, P_m )$, where $P_m$ is realized as $d(x,A_0)$ for some $A_0 \subseteq \bar{a}_m\bar{a}'_m$.  
We may assume that the extended spaces form a Cauchy sequence. 
Furthermore we may assume that there is $j \le 2m$ such that $P_m (a_{m,i})$ is the distance from $a_{m,i}$ to $a_{m,j}$. 
It is now clear that $d(a_{m,i},a_{m,j}) \rightarrow d(a_i ,a_j ) =  P(a_i )$. 
$\Box$ 

\bigskip 

\begin{lem} 
Let $M$ be the Fa\"{i}ss\'{e} limit of $\mathcal{K}$. 
Then $M$ is isometric to a structure of the form $(\mathfrak{U},d,P)$. 
\end{lem} 

{\em Proof.} 
Applying Lemma \ref{Fra} and Theorem 2.19 of \cite{BYFra} we see that $M$ is a separable approximately ultrahomogeneous structure, i.e. every isomorphism of finite substructures is arbitrarily close to the restriction of an automorphism.  
A Polish space $\mathcal{X}$ of diameter 1 is called finitely injective, if given any two finite metric spaces $A\subset A'$ of diameter $\le 1$ any isometric embedding $A \rightarrow \mathcal{X}$ extends to an isometric embedding of $A'$ into $\mathcal{X}$. 
Since $\mathfrak{U}$ is ultrahomogeneous, it is finitely injective and, furthermore it is the only finitely injective Polish space of diameter 1 up to isometry (by uniqueness of an appropriate Fra\"{i}ss\'{e} limit).    
Thus in order to prove the lemma we need to show that $M$ is finitely injective as a metric space of diameter $\le 1$. 
Take $A \subset A'$ and $\phi : A \rightarrow \mathcal{X}$ as above. 
Let $P: A \rightarrow [0,1]$ be the predicate on $A$ induced under $\phi$ by the corresponding predicate on $\mathcal{X}$.  
Let us extend it to a predicate on $A'$. 
Firstly find $(B, B_0)$ such that $A \subseteq B$ and 
$P(x) = d(x,B_0 )$ for $x\in A$. 
Let $C$ be the free amalgamation of $B$ and $A'$ over $A$. 
Define $P(x)$ on $C$ by $P(x) = d(x,B_0 )$ and let $P'$ be the restriction of it on $A'$. 
Then $(A,d,P)$ is a substructure of $(A',d,P')$. 
By Corollary 2.17 of \cite{BYFra}, for any $\varepsilon$ the map $\phi$ is $\varepsilon$-close to an embedding of $(A',d,P')$
into $M$. 

We now see that given any two finite metric spaces $A\subset A'$ of diameter $\le 1$ any isometric embedding $A \rightarrow M$ is arbitrarily close to an isometric embedding of $A'$ into $\mathcal{X}$. 
It is a folklore fact that the latter implies finite injectivity of $M$ as a metric space of diameter 1.  
For example, one can adapt the proof  of Theorem 3.4 of \cite{mel} (in the unbounded case)  
that the approximate extension property implies the extension property in complete spaces. 
$\Box$

\bigskip 

\begin{thm} \label{M-rel-st}
Let $M=(\mathfrak{U},d,P)$ be the Fa\"{i}ss\'{e} limit of $\mathcal{K}$. 

Then $Th(M)$ is separably categorical and the ${\bf t}$-closed set corresponding to $Th(M)$ is relatively stable.  
\end{thm} 

{\em Proof.} 
Let us show that $M$ is approximately oligomorphic, i.e. for any $n$ and any $\varepsilon$ there exists a finite subset, say $\mathsf{F} \subset M^n$ such that $\mathsf{Aut} (M) \mathsf{F}$ is $\varepsilon$-dense in $M^n$. 
It is well-known that approximate oligomorphicity is equivalent to separable categoricity, see \cite{BYBHU}. 

Using the fact that $Th(\mathfrak{U})$ is separably categorical, for every natural $\ell_0 \le n$ find a finite set $\mathsf{F}_{\ell_0} \subset M^{n+\ell_0}$ which has an $\frac{\varepsilon}{8n}$-dense orbit in $M^{n+\ell_0}$ with respect to $\mathsf{Aut} (\mathfrak{U})$. 
Given $\bar{\mathsf{f}} \in  \mathsf{F}_{\ell_0}$ and a sequence $\ell_1 < \ldots <\ell_k$ of numbers $\le n+\ell_0$ with $k\le n$, 
let $(\bar{\mathsf{f}}, \bar{\ell})$ be a pair where $\bar{\ell}$ denotes the tuple  $\ell_1 , \ldots ,\ell_k$. 
We view such $(\bar{\mathsf{f}}, \bar{\ell})$ as a metric structure together with the function 
$P_{\bar{\ell}} (x) = d (x, \{ \mathsf{f}_i \, | \, i \in \bar{\ell}\})$. 
Let $\tilde{\mathsf{F}}$ consist of isomorphism types of ordered structures arising as initial segments of length $n$ of so defined 
$(\bar{\mathsf{f}},d,P_{\bar{\ell}})$, for all $\bar{\mathsf{f}} \in  \bigcup \{ \mathsf{F}_{\ell_0} \, | \, \ell_0\le n\}$ and $\bar{\ell}$ as above. 
Let us show that there is a family of $\mathsf{Aut}(M)$-orbits of structures of types 
of $\tilde{\mathsf{F}}$, which is $\varepsilon$-dense in $M^n$.  

If $(B,d,P) \in \mathcal{K}$ is of size $n$, then for some $\ell_0 \le n$ and a tuple 
$\bar{b} \in \mathfrak{U}^{n+\ell_0}$,  the set $B$ can be viewed as a set consisting of 
$i$-coordinates of $\bar{b}$ with $1\le i \le n$, and $P$ is the distance $d(x,B_0)$, 
where $B_0$ consists of $\ell_i$-coordinates of $\bar{b}$ for some $\ell_1 < \ldots <\ell_k$ as above. 
Let $\bar{\mathsf{f}}'$ be a copy of some $\bar{\mathsf{f}} \in \mathsf{F}_{\ell_0}$ 
which is at distance $\le \frac{\varepsilon}{5n}$ from $\bar{b}$ (as a tuple of the same length).   
Note that if a pair $(\mathsf{f}_i , \mathsf{f}_j )$ of $\bar{\mathsf{f}}'$  corresponds to $(b_i , b_j )$ of $\bar{b}$ then 
\[ 
|d(\mathsf{f}_i , \mathsf{f}_j ) - d(b_i , b_j ) | <  \frac{\varepsilon}{2n }. 
\] 
We now build a finite space extending $\bar{b}$. 
Let us add to $\bar{b}$ a copy of $\bar{\mathsf{f}}$, say $\bar{\mathsf{f}}''$, so that the distance between the corresponding elements with indexes $\le n$ is exactly  $\frac{\varepsilon}{2}$. 
Then extend the metric $d$ for other pairs from $\bar{b}\times \bar{\mathsf{f}}''$ so that 
\[ 
d (x,y) = \mathsf{min} \{ 1, \mathsf{min} \{ d(x,z_1 ) + \frac{\varepsilon}{2} + d(y,z_2) \, | \, z_1 \in \bar{b}  \mbox{ and }  z_2  \in \bar{\mathsf{f}}''  \mbox{ have the same index} \le n \} 
\} . 
\]  
Note that for $(e_1 , e_2 )\in \bar{b}\times \bar{\mathsf{f}}''$ and $e'_2 \in \bar{b}$ of the same index with $e_2 \in\bar{\mathsf{f}}''$, if $(z_1 ,z_2 ) \in\bar{b}\times \bar{\mathsf{f}}''$ is the pair with the index $\le n$ determining $d(e_1 , e_2 )$ as in the definition of $d(x,y)$ above, then $d(e_1 , e'_2 ) \le d(e_1 , z_1 ) + d(z_1 , e'_2 ) \le d(e_1 , e_2 )$. 

Assume that $C_0$ consists of $\ell_i$-coordinates of $\bar{\mathsf{f}}''$ for $\ell_1 < \ldots <\ell_k$ as above, and let 
$P(x)$ be the distance $d(x,B_0 \cup C_0)$ of the space that we have defined.  
The similarity of $\bar{b}$ and $\bar{\mathsf{f}}$ ensures that this distance coincides with $d(x,B_0 )$ and $d(x,C_0 )$ on the corresponding subspaces. 
Indeed, the property which we noted after the definition of $d(x,y)$ (and the dual version of it) supports the corresponding verification. 
As a result, we see that $(B,d,P)$ can be amagamated with the initial segment of length $n$ of some  
structure $(\bar{\mathsf{f}},d,P_{\bar{\ell}})$ as above, so that the distance between the structures is not greater than $\frac{\varepsilon}{2}$. 
Using Corollary 2.17 of \cite{BYFra} and the fact that $M$ is the Fra\"{i}ss\'{e} limit of $\mathcal{K}$, the result of this amalgamation can be embedded into $M$ with the image of $B$ arbitrary close to $B$. 
This implies that the structures of $\tilde{\mathsf{F}}$ witness approximate oligomorphicity of $M$. 

Let us show that for any $\bar{x}$-type $p$, any $\bar{y}$-type $q$ of $Th(\mathfrak{U})$ and any  formula $\phi (\bar{x},\bar{y})$ 
of $Th(M)$ there is a uniform limit of formulas of $Th(\mathfrak{U})$, say $\phi^{p,q} (\bar{x}, \bar{y})$, 
which coincides with $\phi (\bar{x},\bar{y})$ on $p(\bar{x}) \cup q(\bar{y})$.  
By approximate ultrahomogeneity of $M$ we may assume that $\phi (\bar{x},\bar{y})$ is quantifier free. 
 
Take tuples $\bar{s}$ and $\bar{s}'$ realizing $p(\bar{x})$ and $q(\bar{y})$ respectively. 
For any $\varepsilon$ and any $c\in \bar{s}_n \bar{s}'_n$ 
choose a rational $o_{\varepsilon,c}$ with 
$|o_{\varepsilon,c}- P(c)| <\varepsilon$. 
Let $\theta_{\varepsilon} (\bar{x},\bar{y})$ be obtained from $\phi$ by replacing the atomic formulas $P(x_{i})$ 
(resp. $P(y_{j})$) by $o_{\varepsilon,c}$ fo appropriate $c$. 
Let $\phi^{p,q} (\bar{x}, \bar{y})$ be the definable predicate defined by a uniformly convergent sequence 
$\theta_{\varepsilon}$ with $\varepsilon \rightarrow 0$.
It coincides with $\phi (\bar{x}, \bar{y})$ on $p(\bar{x}) \cup q(\bar{y})$. 
Indeed, it is easy to see  that $\phi(\bar{x} ,\bar{y})$ and $\theta_{\varepsilon}(\bar{x} ,\bar{y})$ 
do not differ up to the value of the inverse continuity modulus for $\phi (\bar{x},\bar{y})$ (viewed 
as a function depending on atomic subformulas in $\phi(\bar{x} ,\bar{y})$). 

The existence of $\phi^{p,q} (\bar{x}, \bar{y})$ guarantees that $\phi (\bar{s},\bar{s}')$ 
is stable with respect to $Th(M)$ under Definition \ref{preunst}. 
Indeed, assuming that $\phi (\bar{s},\bar{s}')$ is not stable, find $r_1 < r_2$ witnessing Definition \ref{preunst}. 
Then for sufficiently large $n$ and small $\varepsilon$ find a sequence  $\bar{s}_1 \bar{s}'_1 \ldots \bar{s}_n \bar{s}'_n \in M$ 
as in that definition. 
We can arrange that up to $\varepsilon$ the values of atomic formulas $P(c)$ in  $\phi(\bar{s}_i ,\bar{s}'_j)$ do not depend on $n$ 
(taking a subsequence of $\bar{s}_i\bar{s}'_i$ after enlarging $n$ and reducing $\varepsilon$).   
Since $\phi_{p,q}(\bar{s}_i ,\bar{s}'_j)$ is determined by $tp^{\mathfrak{U}}(\bar{s}_i ,\bar{s}'_j)$ we easily get a contradiction with the definition that $\phi(\bar{s},\bar{s}')$ is unstable (possibly after enlarging $n$ and reducing $\varepsilon$).  
$\Box$ 

\bigskip

\begin{cor} \label{stab eq} {\bf (from the proof)}  
Let $M$ be as in the formulation of Theorem \ref{M-rel-st}, $\phi (\bar{x},\bar{y})$ be any formula of the language of this structure and $p(\bar{x})$ and $q(\bar{y})$ be types of $Th(M)$.  \\ 
(1) Then there is a predicate $\theta (\bar{x},\bar{y})$ definable over $\emptyset$ with respect to $Th(\mathfrak{U})$ 
such that $\phi (\bar{x},\bar{y})$ is stably equivalent to $\theta(\bar{x} ,\bar{y})$ on $p(\bar{x}) \cup q(\bar{y})$, i.e. 
the formula $\phi (\bar{x}, \bar{y})- \theta (\bar{x} ,\bar{y})$ (viewed in $[-1,1]$) does not realize the order property 
(in the standard form) on the set of $\mathfrak{U}$-realizations of $p(\bar{x}) \cup q(\bar{y})$.  \\ 
(2) The condition that $\phi (\bar{x}, \bar{y})$  is stable on $p(\bar{x}) \cup q(\bar{y})$ is equivalent to the condition that $\phi (\bar{x}, \bar{y})$ is constant on $p(\bar{x}) \cup q(\bar{y})$. 
\end{cor} 

\noindent 
{\em Proof.} (1) follows from the proof. \\ 
(2) According to Example 6.4 of \cite{BYTs} a formula of $Th(\mathfrak{U})$ is stable on $p(\bar{x}) \cup q(\bar{y})$ 
if and only if it is constant. 
The proof given there works for definable predicates. 
Now apply the argument of the final part of the proof of Theorem \ref{M-rel-st}. 
$\Box$ 


\subsection{Appendix: stable equivalence and relative forking }  

In Section 3.1 we introduced stable equivalence as an example of relative stability. 
In Corollary \ref{stab eq} it appeared in the case of a concrete structure. 
In this section we give some information concerning the place of this notion in stability theory. 
 
Let $M$ be a continuous metric structure and 
$\phi(\bar{x},\bar{y})$ and $\theta(\bar{x},\bar{y})$ be formulas with parameters from $M$. 
We remind Definition \ref{stable-eq}:  
\begin{quote} 
 $\phi (\bar{x},\bar{y})$ is {\em stably equivalent   to $\theta(\bar{x},\bar{y})$ with respect to} $Th(M)$ if 
the formula $(\phi - \theta )(\bar{x},\bar{y})$ (viewed in $[-1,1]$) is stable.  
\end{quote} 

According to Lemma \ref{stableq} 
stable equivalence is an equivalence relation on the set of formulas. 

\bigskip 

{\bf Relative stability and definability of types}. 
Assume that $\phi (\bar{x},\bar{y})$ is stably equivalent to $\theta(\bar{x},\bar{y})$ with respect to $Th(M)$. 
Let $p(\bar{x}) \in S_{\phi}(M)$ and $\bar{a}$ be its realization. 
Let $q(\bar{x}) \in S_{\theta}(M)$ be the $\theta$-type of $\bar{a}$.  
According to Theorem 3 of \cite{BYGro} there is a definable predicate $\psi (\bar{y})$ over $M$ defining the 
$(\phi - \theta )$-type of $\bar{a}$ over $M$, i.e.  $(\phi - \theta )(\bar{a},\bar{b})= \psi (\bar{b})$ for all $\bar{b} \in M$. 
As a result, we see that the type $p(\bar{x})$ is defined by the sum $\theta(q,\bar{y}) + \psi(\bar{y})$, i.e. 
$\phi (p, \bar{y}) = \theta(q,\bar{y}) + \psi(\bar{y})$. 

\bigskip  

{\bf Existence of relative non-forking extensions}. 
One of crucial properties of forking in stable theories is a smooth transition from its local theory to the global one.  
This is expessed in Lemma 2.18 of \cite{Pillay}. 
One of versions of this lemma can be formulated as follows. 
\begin{quote} 
Let  $\Delta = \{ \delta_1 (\bar{x},\bar{y}_1 ), \ldots , \delta_n (\bar{x}, \bar{y}_n ) \}$ be a finite set of stable formulas, $A$ be a subset of a model $M$ and $p(\bar{x}) \in S_{\Delta}(A)$. 
Then there is a type $q(\bar{x}) \in S_{\Delta}(M)$ 
definable over $acl(A)$ such that $p(\bar{x}) \cup q(\bar{x})$ is consistent. 
\end{quote} 
Using Section 7 \cite{BYU} (see Propositions 7.15 - 7.18) one observes that in the context of continuous model theory  
the statement holds too. 
For example, for $n=1$ this is Proposition 7.15 in \cite{BYU}. 
The following proposition is a straightforward consequence of these observations. 

\begin{prop} 
Let $\Delta_1 = \{ \phi_1 (\bar{x},\bar{y}), \ldots , \phi_n (\bar{x},\bar{y} )\}$ and $p(\bar{x}) \in S_{\Delta_1} (A)$ where $A\subset M$.  
Assume that $\theta(\bar{x},\bar{y})$ is a formula such that each $\phi_i (\bar{x},\bar{y})$ ($i \le n$) is stably equivalent to $\theta(\bar{x},\bar{y})$ with respect to $Th(M)$. 
Assume that $p(\bar{x})$ is consistent with some $q(\bar{x}) \in S_{\theta}(M)$.  
 
Then putting $\Delta_2 = \{ \phi_i - \theta \, | \, 1\le i \le n\}$, 
there are definable predicates $\psi_i (\bar{y})$, $1 \le i \le n$, almost over $A$ which define a $\Delta_2$-type over $M$ such that the type  
$p(\bar{x})\cup q(\bar{x})$ is consistent with the sums 
$\phi_i (\bar{x},\bar{y}) =\theta(q,\bar{y}) +\psi_i(\bar{y})$. 
\end{prop}


\section{The Effros Borel structure of $\mathsf{Iso} (\mathfrak{U})$-spaces}

In this section we consider Borel complexity of some subsets of the Effros space of the logic space. 
These sets correspond to some standard model theoretic notions. 
We mainly concern stability.

\subsection{The Effros space $\mathcal{F} (\mathcal{Y})$}

Given a Polish space $\mathcal{Y}$ let $\mathcal{F} (\mathcal{Y})$ 
denote the set of closed subsets of $\mathcal{Y}$. 
The Effros structure on $\mathcal{F} (\mathcal{Y})$ 
is the Borel space with respect to the $\sigma$-algebra generated by 
the sets 
$$ 
\mathcal{C}_U = \{ D\in \mathcal{F} (\mathcal{Y}): D\cap U \not=\emptyset \}, 
$$ 
for open $U \subseteq \mathcal{Y}$. 
For various $\mathcal{Y}$ this space serves for analysis 
of Borel complexity of families of closed subsets 
(see \cite{KNT} and \cite{RZ} some recent results). 
It is convenient to use the fact that there is a sequence 
of Kuratowski-Ryll-Nardzewski selectors 
$s_n : \mathcal{F} (\mathcal{Y}) \rightarrow \mathcal{Y}$, $n\in \omega$, 
which are Borel functions such that for every non-empty 
$F\in \mathcal{F} (\mathcal{Y})$ the set $\{ s_n (F) ; n\in \omega \}$ 
is dense in $F$. 

Given a Polish group $G$ and a continuous (or Borel) 
action of $G$ on a Polish space $\mathcal{Y}$ one can consider 
the Borel spaces $  \mathcal{F} (G)^n\times \mathcal{F} (\mathcal{Y} )^m$, $m,n \in \omega$. 
In the situation when $\mathcal{Y}$ and $G$ have  grey bases
$\mathcal{B}$ and $\mathcal{R}$ respectively which satisfy the conditions of Section 2.2, 
one can consider $\mathcal{Y}^m$ with respect to the good topology induced by $\mathcal{B}$ (say ${\bf t}$). 
Then many natural Borel  subsets of $G^n \times \mathcal{Y}^m$ can be viewed as elements of  
$\mathcal{F} (G)^n \times \mathcal{F} ((\mathcal{Y}, {\bf t}) )^m$. 
For example, $G$-invariant subsets of $\mathcal{Y}$ will be viewed below as pairs $(G,X)$ where $X\subseteq \mathcal{Y}$ 
satisfies $GX=X$.  

In order to analyse complexity of concrete families of them it is convenient to fix enumerations of the following countable sets 
\[ 
\mathcal{B}(\mathbb{Q}) = \{ (\phi )_{< r} : 
\phi \in \mathcal{B} 
\mbox{ and } r\in \mathbb{Q} \cap [0,1] \} ,    
\]  
\[ 
\mathcal{R}(G_0 ,\mathbb{Q}) =  \{ (H )_{< r} : 
H \in \mathcal{R} 
\mbox{ and } r\in \mathbb{Q} \cap [0,1] \} ,    
\] 
\[ 
\mathcal{I}nv_{\mathcal{B},\mathcal{R}} =  \{ (\phi ,H)  : 
\phi \in \mathcal{B}
\mbox{ is invariant with respect to } 
H \in \mathcal{R}  \}  \, \mbox{ and }
\]  
\[ 
\mathcal{B}_{o}(\mathbb{Q}) = \{ (\phi )_{< r} : \phi 
\mbox{ is a clopen member of } \mathcal{B} 
\mbox{ and } 
r\in \mathbb{Q} \cap [0,1] \} .     
\]  
In the case of the logic space $\mathcal{Y}=\mathfrak{U}_L$ as in Section 2.3, ${\bf t}$-closed invariant subsets are identified with theories in $\mathcal{L}$. 
Then we use the following notation: $\mathcal{B}_{\mathcal{L}}(\mathbb{Q})$, $\mathcal{R}^{\mathfrak{U}}(G_0 ,\mathbb{Q})$, 
$\mathcal{I}nv_{\mathcal{B},\mathcal{R}}$ and $\mathcal{B}_{o\mathcal{L}}$.  
Note that in this case the latter family of the list is a basis of the topology of $\mathcal{Y}$, denoted by $\tau$ in Section 2. 

\begin{remark} 
{\em The setup in the standard case of first-order theories is as follows. 
Having a language $L$ the space of all complete theories is considered under the topology generated by clopen sets defined by first-order formulas without parameters. 
This corresponds to the family $\mathcal{B}(\mathbb{Q})$. 
When one considers the logic space $X_L= \prod_{i\in I} 2^{\omega^{n_i}}$ some counterparts of the remaining families 
of the list naturally appear. 
For example $\mathcal{R}(G_0 ,\mathbb{Q})$ appears as the family of all stabilizers of finite subsets of $\omega$. }
\end{remark} 

By Theorem 2.2 of \cite{CL} for any Polish group $G$ 
and any standard Borel $G$-space $\mathcal{X}$
there is a continuous group monomorphism 
$\Phi : G \rightarrow \mathsf{Iso} (\mathfrak{U} )$ 
and a Borel $\Phi$-equivariant injection $f: \mathcal{X} \rightarrow \mathfrak{U}_L$. 
We only need here that the language $L$ is countable relational 
with 1-Lipschitz symbols of unbounded arity.  
As a result all Polish groups can be considered as elements of 
$\mathcal{F} (\mathsf{Iso} (\mathfrak{U} ))$, all Polish spaces 
are elements of $\mathcal{F} (\mathfrak{U}_L)$ and Polish $G$-spaces 
are pairs from 
$\mathcal{F} (\mathsf{Iso} (\mathfrak{U} )) \times \mathcal{F} (\mathfrak{U}_L )$ (were the action is by automorphisms).  
This explains why the case of the platform $\mathfrak{U}$ is basic for us. 

Let $\mathcal{R}^{\mathfrak{U}}(G_0 )$ and $\mathcal{B}_{\mathcal{L}}$ be the grey bases defined in Section \ref{CAS} in the case of 
$\mathsf{Iso}( \mathfrak{U} )$ and $\mathfrak{U}$.  
Let ${\bf t}$ be the corresponding good topology (in fact, it is a nice topology). 
The following proposition is a version of a well-known fact. 

\begin{prop} \label{elem-Borel} 
(1) The following relations  from  $(\mathcal{F} (\mathsf{Iso} (\mathfrak{U})))^2$,  
$(\mathcal{F} (\mathfrak{U}_L ))^2$,  $(\mathcal{F} (\mathfrak{U}_L ,{\bf t}))^2$,
$(\mathcal{F} (\mathsf{Iso} (\mathfrak{U} )))^3$, 
$\mathcal{F} (\mathsf{Iso} (\mathfrak{U} ))\times \mathcal{F} (\mathfrak{U}_L ) \times \mathcal{F} (\mathfrak{U}_L )$ and 
$\mathcal{F} (\mathsf{Iso} (\mathfrak{U} ))\times \mathcal{F} (\mathfrak{U}_L ,{\bf t}) \times \mathcal{F} (\mathfrak{U}_L, {\bf t} )$ 
(under natural interpretations) are Borel: 
\[ 
\{ (A,B): A \subseteq B \} \mbox{ , }  \{ (A,B,C): AB \subseteq C\}. 
\]  
(2) The closed subgroups of $\mathsf{Iso} (\mathfrak{U})$ 
form a Borel set $\mathcal{U} (\mathsf{Iso} (\mathfrak{U}))$ in 
$\mathcal{F} (\mathsf{Iso} (\mathfrak{U}))$. \\
(3) The Polish $G$-spaces form a Borel set in 
$ \mathcal{F} (\mathsf{Iso} (\mathfrak{U} ))\times \mathcal{F} (\mathfrak{U}_L )$  
and the closed $G$-subspaces of $(\mathfrak{U}_L ,{\bf t})$ form a Borel set in 
$ \mathcal{F} (\mathsf{Iso} (\mathfrak{U} ))\times\mathcal{F} (\mathfrak{U}_L ,{\bf t})$. 
\end{prop} 

{\em Proof.} Statement (2) is well-known: see Section 3.2 of \cite{RZ}.  
Furthermore, statements (1) and (2) are variants of Lemmas 2.4 and 2.5 from \cite{KNT} which were proved for  $S_{\infty}$. 
It is also  mentioned in \cite{KNT} that they hold in general. 
Statement (3) follows  from (1) and (2). 
$\Box$ 

\bigskip 

We say that an $\mathsf{Iso} (\mathfrak{U})$-invariant set $X\in \mathcal{F} (\mathfrak{U}_L ,{\bf t})$ is {\em indecomposable $\mathsf{Iso} (\mathfrak{U})$-invariant} if for any $\mathsf{Iso} (\mathfrak{U})$-invariant $\phi \in \mathcal{B}_{\mathcal{L}}$ and any $r \in [0,1]$ the set $X$ is a subset of $(\phi )_{<r}$ or its complement. 
Clearly this corresponds to the notion of complete theories.  
Note that by Proposition \ref{elem-Borel} the pairs $(\mathsf{Iso} (\mathfrak{U}), X)$ with $\mathsf{Iso} (\mathfrak{U})$-invariant $X$, 
form a Borel family. 
The following statement is the starting point of our analysis. 

\begin{prop} \label{complete} 
The set of indecomposable $G$-invariant members of 
$\mathcal{F} (\mathfrak{U}_L ,{\bf t})$ is Borel. 
\end{prop} 

{\em Proof.} 
By the definition of indecomposable members, for a $G$-invariant $X$ 
\[ 
X \mbox{ is indecomposable } \Leftrightarrow 
(\forall B\in \mathcal{B}_{\mathcal{L}}(\mathbb{Q}))
(X \subseteq B \vee X\cap B = \emptyset ). 
\]  
By Proposition \ref{elem-Borel} this is a definition of a Borel 
family. 
$\Box$ 
\bigskip 

\begin{remark} \label{Borel4X}
{\em This proposition and the notions used in it have natural generalizations in the general case, when we just have a $(G, \mathcal{R})$-space $\mathcal{X}$, a good basis $\mathcal{B}$, and analyze  {\bf t}-closed subsets $Y\subseteq \mathcal{X}$. 
As above indecomposable $G$-invariant members of $\mathcal{F} (\mathcal{X} ,{\bf t})$ form a Borel family. 
To see this note, that the condition $GY =Y$ can be expressed in a Borel way even in these general circumstances. 
Indeed, it is easy to see that for a Kuratowski-Ryll-Nardzewski's selector $s_l$ the condition $s_l (X_1 ) \not\in X_2$ defines a Borel subset of 
$\mathcal{F}(\mathcal{X})\times \mathcal{F}(\mathcal{X})$.  
Using this one can easily express the statement that $Y$ and $GY$ are dense in each other. }
\end{remark}

\subsection{Model companions}

In model theory one of the definitions of model completeness  states that any formula is equivalent to an existential one (or a universal one). 
A theory $T_1$ is a model companion of $T$ 
if $T_1$ is model complete and every model of $T$ embeds into a model of $T_1$ and vice versa.

The following definition is a variant of Definition 1.3 from \cite{MI05} 
and Proposition 1.4 of \cite{IMI-APAL}. 
 
\begin{definition} 
Given a $(G, \mathcal{R})$-space $(\mathcal{X}, \tau)$ (as in Section 2.2), and an $\mathcal{R}$-good basis $\mathcal{B}$,  
assume that  $X_0$ and $X_1$ are  {\bf t}-closed invariant subsets of $\mathcal{X}$.  
We say that $X_1$ is a {\sl companion} of $X_0$ if the $\tau$-closures of 
$X_0$ and $X_1$ coincide and every element of $\mathcal{B}$  
is $\tau$-clopen on $X_1$. 
\end{definition} 
To see that this definition is a version of the classical definition, the family $\mathcal{B}(\mathbb{Q})$ and 
the subfamily of its $\tau$-closed members should be viewed as the set of all formulas and the subset of all universal ones. 
In particular, the definition says that $X_0$ and $X_1$ cannot be distinguished by universal sentences and in $X_1$ every formula is equivalent to a universal one. 

We now consider the case when $\mathcal{X} = \mathfrak{U}_L$ and $G= \mathsf{Iso}(\mathfrak{U})$ are taken together with the nice basis 
$\mathcal{B}_{\mathcal{L}}$ defined in Section 2.3. 
 
\begin{thm} \label{comp}
The set of pairs $(X_0, X_1 )$ of 
$\, \mathsf{Iso} (\mathfrak{U} )$-invariant members of 
$\mathcal{F} (\mathfrak{U}_L ,{\bf t})$ 
with the condition that $X_1$ is a companion of $X_0$  is Borel. 
\end{thm} 

{\em Proof.} 
Applying Proposition \ref{elem-Borel} we consider pairs $(X_0 ,X_1 )$  
of ${\bf t}$-closed invariant subsets 
as elements of the corresponding Borel set of triples 
$(\mathsf{Iso} (\mathfrak{U} ), X_0 , X_1 )$.  
Using Kuratowski-Ryll-Nardzewski selectors the condition that $\tau$-closures 
of $X_0$ and $X_1$ are the same can be written as follows:  
\[ 
(\forall A_k \in \mathcal{B}_{o\mathcal{L}}(\mathbb{Q}))\forall i \exists j \exists l ( 
s_i (X_1 ) \in A_k \rightarrow s_j (X_0 ) \in A_k ) \wedge 
( s_i (X_0 ) \in A_k \rightarrow s_l (X_1 ) \in A_k ). 
\]  
Now note, that given two ${\bf t}$-closed $A$ and $B$ 
the condition $A\cap X_1 \subseteq B\cap X_1$ for $X_1 \in \mathcal{F}(\mathcal{X})$  
is equivalent to the formula: 
\[  
\forall j (s_j (X_1 )\in A \rightarrow s_j (X_1 ) \in B ). 
\]  
In particular, this condition is Borel. 
We can now express that every element of $\mathcal{B}_{\mathcal{L}}(\mathbb{Q})$ 
is $\tau$-clopen on $X_1$ as follows: 
\[  
 (\forall B_l \in \mathcal{B}_{\mathcal{L}}(\mathbb{Q})) 
(\exists A_k \in \mathcal{B}_{o\mathcal{L}}(\mathbb{Q})) 
(\exists A_m \in \mathcal{B}_{o\mathcal{L}}(\mathbb{Q})) 
( (B_l \cap X_1\subseteq A_k \cap X_1 ) \wedge 
\] 
\[ 
(A_k \cap X_1\subseteq B_l \cap X_1)) \wedge 
(A_m \cap X_1 = X_1 \setminus B_l ). 
\] 
For the final equality of this formula note that $\forall i (s_i (X_1 ) \not\in B_l \rightarrow s_i (X_1) \in A_m )$ means 
$X_1 \setminus B_l \subseteq A_m \cap X_1$.  
$\Box$ 

\bigskip 

\begin{remark} 
{\em The theorem above is a counterpart of the statement that identifying theories of a language $L$ with closed subsets of the compact space of complete $L$-theories, 
the binary relation to be a model companion is Borel. 
Although the author has not found it in the literature, it is an easy exercise. }
\end{remark}

\subsection{A Borel substitute for stability} 
Preserving the notation of Section 4.1 we give a very general definition of stability. 
It admits a description of stable ${\bf t}$-closed sets as a Borel family. 
We will also see that under the circumstances  of  the beginning of Section 3 (i.e. when 
a {\bf t}-closed $Y \subseteq \mathcal{M}_{L}$ is considered) 
it is equivalent to  Definition \ref{preunst} under the additional assumption that $\mathcal{M}$ is separably categorical.

We remind the reader that in the general case of Section 4.1 we have a $(G, \mathcal{R})$-space $\mathcal{X}$, 
a dense subgroup $G_0 \le G$, a good basis $\mathcal{B}$, a grey subset $\phi \in \mathcal{B}$ and 
a {\bf t}-closed subset $Y\subseteq \mathcal{X}$. 
Instead of $\bar{s}$ and $\bar{s}'$ grey subgroups $H, H'\in \mathcal{R}$ are given such that $\phi$ 
is invariant with  respect to $\mathsf{max} (H, H' )$. 

\begin{definition} \label{unst} 
A ${\bf t}$-closed set $Y$ is {\sl  unstable relatively to} $(\mathcal{X},\mathcal{B})$ if there is a grey set $\phi \in \mathcal{B}$ such that for every $n$ and every $\varepsilon\in \mathbb{Q} \cap [0,1]$ there exist 
\[ 
g_{1,1}, g_{1,2}, \ldots ,g_{1,n}, g_{2,1}, \ldots ,g_{n-1,n}, g_{1,n},\ldots ,g_{n,n} \in G_0
\]   
such that $H(g^{-1}_{i,j} g_{i,l})\le \varepsilon$ and $H'(g^{-1}_{i,j}g_{l,j})\le \varepsilon$ for all $i,j,l\in \{ 1,\ldots ,n\}$  and  
\[ 
Y \cap \bigcap \{ (g_{i,j} \phi )_{\le \varepsilon} : i < j \}  \cap  \bigcap \{ (g_{i,j} \phi )_{\ge 1 - \varepsilon} : j \le i \} \not=\emptyset . 
\] 
\end{definition}  

We now show that this definition is a counterpart of Definition \ref{preunst}. 
W.l.o.g. assume that $\phi$ and all $L$-symbols below are continuous with respect to $\mathsf{id}$.

\begin{prop} \label{preunst-1-2-3} 
Assume that $\mathcal{M}$ is a separably categorical first-order structure, $N$ is approximating in $\mathcal{M}$ and 
$Y$ is defined by a first-order $L$-theory. 
Then Definition \ref{unst} is equivalent to existence of a continuous first-order formula $\phi (\bar{s},\bar{s}')$ with parameters from $N$ such that Definition \ref{preunst} holds.  
\end{prop} 

{\em Proof.} 
Assume that Definition \ref{unst} is satisfied by a grey subset $\phi (\bar{s},\bar{s}')$ invariant with respect to maximum of $H=H_{q, \bar{s}}$ and $H'=H_{q',\bar{s}'}$. 
Assume for simplicity that $q=q'=1$. 

Let $r_1 = \frac{1}{4}$ and $r_2 = \frac{3}{4}$. 
Take any $n$ and $\varepsilon <\frac{1}{8}$ and find the corresponding 
\[ 
g_{1,1}, g_{1,2}, \ldots ,g_{1,n}, g_{2,1}, \ldots ,g_{n-1,n}, g_{1,n},\ldots ,g_{n,n} \in G_0 . 
\]    
The tuples 
$\bar{s}_1 ,\bar{s}'_1 ,\ldots ,\bar{s}_n ,\bar{s}'_n$ as in Definition \ref{preunst} are taken as images of $\bar{s}$, $\bar{s}'$ with respect to $g_{1,1}, g_{2,2}, g_{3,3}, \ldots , g_{n-1,n-1}, g_{n,n} \in G_0$. 
Note that for any $i$ and $j$ we have 
$d(\bar{s}_i \bar{s}'_j, g_{i,j} (\bar{s}\bar{s}')) \le \varepsilon$,   i.e. 
\[ 
d(tp^{\mathcal{M}}(\bar{s}_i \bar{s}'_j), tp^{\mathcal{M}}(\bar{s} \bar{s}'))\le \varepsilon \mbox{ , and } 
\] 
\[ 
| g_{i,j} \phi (\bar{s},\bar{s}') - \phi (\bar{s}_i ,\bar{s}'_j)| \le \varepsilon . 
\]  
In particular we now see that the structure from $Y$ satisfying the final inequality of Definition \ref{unst} also satisfies the  corresponding inequality of Definition \ref{preunst}. 

To prove the contrary direction assume that 
there exists a continuous first-order formula $\phi' (\bar{s},\bar{s}')$ with parameters from $N$ such that Definition \ref{preunst} holds. 
Fix $r_1$ and $r_2$. 
Let $\phi$ be $\frac{\phi' (\bar{x},\bar{x}')\dot{-} r_1}{r_2 - r_1}$.  
Now take sufficiently small $\varepsilon$ and sufficiently large $n$. 
Applying Definition \ref{preunst} to $\phi'$, $n$ and $\varepsilon$ for every $i,j$ let $g_{i,j}$ be an isometry such that up to $\varepsilon$ it maps $\bar{s}$ to $\bar{s}_i$ and $\bar{s}'$ to $\bar{s}'_j$ respectively. 
In order to find these isometries use approximate oligomorphicity,  see Theorem 12.10 of \cite{BYBHU}. 
It is now easy to see that Definition \ref{unst} holds for $\phi$.  
$\Box$ 

\bigskip

We now say that the ${\bf t}$-closed $Y$ is {\em stable relatively to} $(\mathcal{X},\mathcal{B})$ if it is not unstable relatively to $(\mathcal{X},\mathcal{B})$.

\begin{prop} 
Under the circumstances of Definition \ref{unst} the family 

$\{ Y \in \mathcal{F}(\mathcal{X},{\bf t}) \, | \, Y$ is $G$-invariant and stable relatively to $(\mathcal{X},\mathcal{B}) \, \}$ is Borel. 
\end{prop}

{\em Proof.} 
We consider ${\bf t}$-closed invariant subsets $Y$  
as elements of the corresponding Borel set of pairs  $(G, Y )$.  
We may also use Remark \ref{Borel4X}. 
Since the sets $\mathbb{Q} \cap [0,1]$, $G_0$, 
$\mathcal{B}(\mathbb{Q})$,  $\mathcal{R}(G_0 ,\mathbb{Q})$
and $\mathcal{I}nv_{\mathcal{B},\mathcal{R}}$ are countable 
it suffices to show that given 
$g_{1,1}, g_{1,2}, \ldots ,g_{1,n}, g_{2,1}, \ldots ,g_{n-1,n}, g_{1,n},\ldots ,g_{n,n} \in G_0$  and $(\phi )_{\le \varepsilon}$ and 
$(\phi )_{\ge 1 -\varepsilon} \in \mathcal{B}(\mathbb{Q})$ 
the family of all ${\bf t}$-closed $G$-invariant $Y$ with   
\[ 
Y \cap \bigcap \{ (g_{i,j} \phi )_{\le \varepsilon} : i < j \}  \cap  \bigcap \{ (g_{i,j} \phi )_{\ge 1 - \varepsilon} : j \le i \} \not=\emptyset 
\] 
is Borel. 
Note that this condition can be written as a statement of existence of a Kuratowski-Ryll-Nardzewski selectors $\mathsf{s}_{\ell}$ such that the conjunction of the following statements holds.   
\[  
g^{-1}_{i,j} (s_{\ell} (Y)) \in (\phi )_{\le \varepsilon} \mbox{ for  }  i < j \mbox{ and } 
g^{-1}_{i,j} (s_{\ell} (Y)) \in (\phi )_{\le 1- \varepsilon} \mbox{ for  }  j \le i . 
\]  
The rest follows from the continuity of the action of $G$ on $\mathcal{X}$. 
$\Box$

\subsection{NIP} 

Let us develope the approach of Sections 4.1 and 4.3 in the case of some other neostability properties. 
We preserve the situation of the general case of Section 4.1, i.e.  we have a $(G, \mathcal{R})$-space $\mathcal{X}$, 
a dense subgroup $G_0 \le G$, a good basis $\mathcal{B}$,  and a {\bf t}-closed subset $Y\subseteq \mathcal{X}$.

\begin{definition} \label{ip} 
A ${\bf t}$-closed set $Y$ has the {\sl  independence property relatively to} $(\mathcal{X},\mathcal{B})$ if there is a grey set $\phi \in \mathcal{B}$ together with grey subgroups $H, H'\in \mathcal{R}$ such that $\phi$ is invariant with  respect to $\mathsf{max} (H, H' )$
and the following property holds: for any $n$ and any $\varepsilon\in \mathbb{Q} \cap [0,1]$ there exist 
\[ 
\{ g_{1,I} \, | \, I \subseteq \{ 1, 2, \ldots n\} \} \cup \{ g_{2,I} \, | \, I\subseteq \{ 1,2,  \ldots n\} \} \cup \ldots \cup \{ g_{n,I} \, | \, I \subseteq \{ 1,2,\ldots ,n\}\}  \subseteq  G_0
\]   
such that $H(g^{-1}_{i,I} g_{i,J})\le \varepsilon$ and $H'(g^{-1}_{i,I}g_{j,I})\le \varepsilon$ for all $i,j\in \{ 1,\ldots ,n\}$, $I, J \subseteq \{ 1,\ldots , n\}$,   and  
\[ 
Y \cap \bigcap \{ (g_{i,I} \phi )_{\le \varepsilon} : i \in I \}  \cap  \bigcap \{ (g_{i,I} \phi )_{\ge 1 - \varepsilon} :  i \not\in I\} \not=\emptyset . 
\] 
\end{definition}  

\begin{prop} \label{preip} 
Assume that $\mathcal{M}$ is a separably categorical first-order structure, $N$ is approximating in $\mathcal{M}$ and 
$Y$ is defined by a first-order $L$-theory. 
Then Definition \ref{ip} is equivalent to existence of a continuous first-order formula $\varphi (\bar{s},\bar{s}')$ with parameters from $N$ such that  for every $n$ and $\varepsilon \in (0,1)$ the formula $\varphi (\bar{x},\bar{x}')$ has the independence property (IP) in the following form: there are 
$( \bar{a}_i )_{i\le n}$  and $(\bar{b}_I )_{I\subseteq [n]}$ in $\mathcal{M}$ and some $M\in Y$ such that 
\[ 
d(tp^{\mathcal{M}}(\bar{a}_i \bar{b}_I), tp^{\mathcal{M}}(\bar{s} \bar{s}'))\le \varepsilon \, , \, i\le n \, , \, I\subseteq [n] \,  \mbox{, and }   
\]  
\[ 
M\models (\varphi (\bar{a}_i , \bar{b}_I ))_{<\varepsilon} \Leftrightarrow i\in I \mbox{ , } 
\] 
\[  
M\models (\varphi (\bar{a}_i , \bar{b}_I ))_{> 1 -\varepsilon} \Leftrightarrow i\not\in I
\]  
\end{prop} 

{\em Proof.} 
Assume that Definition \ref{ip} is satisfied by a grey subset $\phi (\bar{s},\bar{s}')$ invariant with respect to maximum of $H=H_{q, \bar{s}}$ and $H'=H_{q',\bar{s}'}$. 
For simplicity we may assume that $q = q' = 1$. 

Take any $n$ and $\varepsilon <\frac{1}{4}$. 
Find   
\[ 
\{ g_{1,I} \, | \, I \subseteq \{ 1, 2, \ldots n\} \} \cup \{ g_{2,I} \, | \, I\subseteq \{ 1,2,  \ldots n\} \} \cup \ldots \cup \{ g_{n,I} \, | \, I \subseteq \{ 1,2,\ldots ,n\}\}  \subseteq  G_0 . 
\]   
as in Definition \ref{ip} corresponding to $\varepsilon$. 
Choose $\bar{a}_i $,  $ \, i \in \{ 1, 2, \ldots n\}  \, $, as the images of $\bar{s}$ with respect to $g_{1,\emptyset }, g_{2,\emptyset }, \ldots , g_{n-1 ,\emptyset }, g_{n,\emptyset }$. 
Then the tuples 
$\bar{b}_I$,  $ \, I\subseteq \{ 1,2,  \ldots n\} $, corresponding to  the independence property from the formulation, 
are taken as the images of $\bar{s}'$ with respect to  $g_{1,I}$, $ \, I\subseteq \{ 1,2,  \ldots n\} $. 
Note that for any $i$ and $I$ we have 
$d(\bar{a}_i \bar{b}_I , g_{i,I} (\bar{s}\bar{s}')) = d(g_{i,I} g^{-1}_{i,I} (g_{i,\emptyset } (\bar{s}) g_{1, I} (\bar{s}')), g_{i,I} (\bar{s}\bar{s}')) \le \varepsilon$,   i.e. 
\[ 
d(tp^{\mathcal{M}}(\bar{a}_i \bar{b}_I), tp^{\mathcal{M}}(\bar{s} \bar{s}'))\le \varepsilon \mbox{, and }   
\]  
\[ 
| g_{i,I} \phi (\bar{s},\bar{s}') - \phi (\bar{a}_i ,\bar{b}_I)| \le \varepsilon . 
\]  
In particular, we now see that the structure from $Y$ satisfying the final inequality of Definition \ref{ip} also satisfies the  corresponding inequalities of the formulation of the proposition. 

To prove the contrary direction assume that there exists a continuous first-order formula $\phi (\bar{s},\bar{s}')$ with parameters from $N$ such that the final condition of the formulation holds. 
Now take sufficiently small $\varepsilon$ and sufficiently large $n$. 
Applying the condition to $\phi$, $n$ and $\varepsilon$, for any $i\le n$ and $I\subseteq  \{ 1, 2, \ldots n\} $ let $g_{i,I}$ be an isometry such that up to $\varepsilon$ it maps $\bar{s}$ to $\bar{a}_i$ and $\bar{s}'$ to $\bar{b}_I$ respectively. 
In order to find these isometries use approximate oligomorphicity. 
It is now easy to see that Definition \ref{ip} holds for $\phi(\bar{s},\bar{s}')$.  
$\Box$ 

\bigskip

A theory is NIP (dependent) if no formula has the independence property. 
Note, that stability relatively to $(\mathcal{X},\mathcal{B})$ implies NIP with respect to $(\mathcal{X},\mathcal{B})$. 
In particular, the example of Section 3.2 has relative NIP.  

\bigskip

\begin{prop} 
Under the circumstances of Definition \ref{ip} the family \\ 
$\{ Y \in \mathcal{F}(\mathcal{X},{\bf t}) \, | \, Y$ is $G$-invariant and NIP relatively to $(\mathcal{X},\mathcal{B}) \, \}$ is Borel. 
\end{prop}

{\em Proof.} 
We consider ${\bf t}$-closed invariant subsets $Y$  
as elements of the corresponding Borel set of pairs  
$(G, Y )$.  
Since the sets $\mathbb{Q} \cap [0,1]$, $G_0$, 
$\mathcal{B}(\mathbb{Q})$,  $\mathcal{R}(G_0 ,\mathbb{Q})$
and $\mathcal{I}nv_{\mathcal{B},\mathcal{R}}$ are countable 
it suffices to show that given 
\[ 
\{ g_{1,I} \, | \, I \subseteq \{ 1, 2, \ldots n\} \} \cup \{ g_{2,I} \, | \, I\subseteq \{ 1,2,  \ldots n\} \} \cup \ldots \cup \{ g_{n,I} \, | \, I \subseteq \{ 1,2,\ldots ,n\}\}  \subseteq  G_0
\]   
and $(\phi )_{\le \varepsilon}$ and 
$(\phi )_{\ge 1 -\varepsilon} \in \mathcal{B}(\mathbb{Q})$ 
the family of all ${\bf t}$-closed $G$-invariant $Y$ with   
\[ 
Y \cap \bigcap \{ (g_{i,I} \phi )_{\le \varepsilon} : i \in I \}  \cap  \bigcap \{ (g_{i,I} \phi )_{\ge 1 - \varepsilon} : i \not\in I \} \not=\emptyset 
\] 
is Borel. 
Note that this condition can be written as a statement of existence of a Kuratowski-Ryll-Nardzewski selectors $\mathsf{s}_{\ell}$ such that the conjunction of the following statements holds.   
\[  
g^{-1}_{i,I} (s_{\ell} (Y)) \in (\phi )_{\le \varepsilon} \mbox{ for  }  i \in I \mbox{ and } 
g^{-1}_{i,I} (s_{\ell} (Y)) \in (\phi )_{\le 1- \varepsilon} \mbox{ for  }  i \not\in I. 
\]  
The rest follows from continuity of the action of $G$ on $\mathcal{X}$. 
$\Box$

\bigskip

\bigskip 

Institute of Computer Science, \parskip0pt 

University of Opole, ul.Oleska 48, 45 - 052 Opole, Poland \parskip0pt 

E-mail: aleksander.iwanow@uni.opole.pl   
\end{document}